\documentclass[sn-mathphys]{sn-jnl}

\jyear{2021}%

\theoremstyle{thmstyleone}%

\theoremstyle{thmstyletwo}%
\newtheorem{example}{Example}%

\theoremstyle{thmstylethree}%
\usepackage{graphicx}
\usepackage[numbers]{natbib}

\newtheorem{thm}{Theorem}
\newtheorem{cor}[thm]{Corollary}
\newtheorem{lem}[thm]{Lemma}
\newtheorem{prop}[thm]{Proposition}
\newtheorem{defin}[thm]{Definition}

\newtheorem{assume}[thm]{Assumption}
\newtheorem{fact}[thm]{Fact}
\newtheorem{rem}[thm]{Remark}

\DeclareMathOperator*{\dom}{dom}
\DeclareMathOperator*{\Fix}{Fix}

\newcommand{\RR}{ \mathbb{R} }
\newcommand{\NN}{\mathbb{N}}

\newcommand{\prox}{\mathrm{prox}} 

\begin{document}
\title[Trilevel and Multilevel Optimization]{Trilevel and Multilevel Optimization using~Monotone~Operator~Theory}


\author{\fnm{Allahkaram} \sur{Shafiei}}\email{shafiall@fel.cvut.cz}

\author{\fnm{Vyacheslav} \sur{Kungurtsev}}\email{vyacheslav.kungurtsev@fel.cvut.cz}

\author{\fnm{Jakub} \sur{Marecek}}\email{jakub.marecek@fel.cvut.cz}

\affil{\orgdiv{Department of Computer Science}, \orgname{Czech Technical University}, \orgaddress{\street{Karlovo Namesti 13}, \city{Prague 2}, \postcode{121 35},  \country{Czech Republic}}}


\abstract{We consider rather a general class of multi-level optimization problems,
where a convex objective function is to be minimized subject to constraints of optimality of nested convex optimization problems. 
As a special case, we consider a trilevel optimization problem, where the objective of the two lower layers
consists of a sum of a smooth and a non-smooth term.~Based on fixed-point theory and related arguments, we present a natural first-order algorithm and analyze its convergence and rates of convergence in several regimes of parameters.}




\maketitle
\section{Introduction}\label{sec1}

 Hierarchical Optimization Problems, also known as Multilevel Optimization Problems (MOP), were first introduced by \cite{bracken1973mathematical} and \cite{candler1977multi} as a class of constrained optimization problems, wherein the feasible set is  determined -- implicitly -- as the optima of multiple optimization problems, nested in a predetermined sequence. 
In theory, MOP has applications in game theory, robust optimization, chance-constrained programming, and adversarial machine learning.
In practice, MOP models are widely used in security applications, where they model so-called interdiction problems. See 
\cite{iiduka2011iterative,lampariello2020explicit,moudafi2007krasnoselski,xu2010viscosity,yamada2001hybrid}
for several examples.  


\textcolor{black}{
Tri-level programming problems are challenging,  even when one considers continuous linear problems \cite{blair1992computational} due to their computational complexity and the interactions between decision-makers based on the numbers of variables at different levels. Moreover, at each level, we have limited or incomplete information about the decisions made at the other levels. This interdependency makes it difficult to decouple the optimization problem into separate sub-problems. This, in turn, requires techniques that can handle the hierarchical nature of the problem. \\
The article is structured as follows. Section 1 introduces the necessary notations, assumptions, and the problem model. Additionally, it establishes fundamental notation and background information pertaining to the proximal-gradient algorithm. Section 2 delves into  convergence analysis for a variety of assumptions on the step sizes. Furthermore, in this section, we present the methodological approach employed in our study regarding error bounds that enables us to provide the convergence of sequence generated by the algorithm for tri-level problems. In Section 3, the paper introduces the convergence rate analysis for variants of the proximal-gradient algorithm. It asserts that our convergence rate in the middle layer, which is $\mathcal{O}(\frac{1}{(k+1)\sqrt{k}})$, improves upon the rate of  \cite{Sabach2017}. Finally, in Section 4, we formalize multi-level optimization, followed by an exploration of both convergence and the corresponding convergence rate. 
}

\subsection{The Problem}

In particular, our goal is to formulate and analyze an optimization algorithm for a class of hierarchically-defined problems: 
   \begin{equation}\label{prob:multi-level}
\left\{ \begin{array}{l}
 \mathop {\min }\limits_{x \in X_N^*} \omega(x) \\ 
 X_{i}^* = \mathop {\arg \min }\limits_{x\in X_{i-1}^*} [f_i(x) + g_i(x)],\,\,\,\,\,\,i \in \{1,...,N\}  \\ 
 X_0^* = \mathbb{R}^n, \\ 
 \end{array} \right.\end{equation}
where the middle layers (for $i \in \{1,..., N\}$) exhibit the so-called composite structure, where $\omega$ is a strongly convex differentiable function and there are smooth terms $f_i$ and non-smooth terms $g_i$. 
In machine-learning applications, the smooth functions are chosen to be loss functions and the  
non-smooth functions $g_i$ are regularizers.

We begin by considering a hierarchical optimization with three layers, wherein the middle and lower layers exhibit 
the composite structure: 
 \textcolor{black}{ \begin{equation}\label{prob:tri-levelintro}
\left\{ \begin{array}{l}
 \mathop {\min }\limits_{x \in {X^*}} \ \omega (x) \\ 
  {X^*} = \mathop {\arg \min }\limits_{x \in {Y^*}} [\phi_2(x):={f_2}(x) + {g_2}(x)] \\ 
 {Y^*} = \mathop {\arg \min }\limits_{x \in \mathbb{R}^n } [\phi_1(x):={f_1}(x)+g_1(x) ]. \\ 
\end{array} \right.
\end{equation}
By leveraging
fixed-point theory and related reasoning, we propose a straightforward first-order algorithm and analyze its convergence and convergence rates across various parameter regimes. The algorithm exhibits the following non-asymptotic behaviour:
The first layer exhibits a convergence rate of $\mathcal{O}(\frac{1}{k})$, the second layer (middle layer) exhibits a convergence rate $\mathcal{O}(\frac{1}{(k+1)\sqrt{k}})$, and finally, the third layer exhibits $\mathcal{O}(\frac{1}{\sqrt{k}})$ global
rate of convergence concerning the inner objective function values. By accessing the main iteration in terms of the inner objective function values, we observe the convergence rate $\mathcal{O}(\frac{1}{k})$.}\\


 \subsection{Related Work}

Our work is inspired by a long history of work on bilevel optimization problems (see, ~e.g,~\cite{al1992global,dempe2007new,dempe2014necessary,zhang1994problems}).
Our work extends proximal-gradient optimization algorithms \cite{Sabach2017} for a related bilevel optimization problem and is informed by \cite{moudafi2007krasnoselski}.

Notably, \cite{solodov2007explicit} gave an explicit descent method for bi-level optimization in the form of
\begin{equation}\label{r2}
\left\{ \begin{array}{l}
 \min \,\omega(x) \\ 
 x \in S := \mathop {\arg \min }\limits_{x \in F} f(x) = \mathop {\arg \min }\limits_{x \in {\mathbb{R}^n}} [f(x) + {i_F}], \\ 
 \end{array} \right.
\end{equation}
in which $i_F$ is the indicator function on $F$ and $f$ is convex and smooth function and $\omega$ is strongly convex. Subsequently, \cite{Sabach2017} proposed the so-called BIG-SAM method for solving the more general problem of,
\begin{equation}\label{eq:bilevelproblem}
\left\{ \begin{array}{l}
 \min \,\omega(x) \\ 
 x \in {Y^*} := \mathop {\arg \min }\limits_{x \in {\mathbb{R}^n}} [f(x) + g(x)], \\ 
 \end{array} \right.
\end{equation}
 in which $f$ is smooth and $g$ is convex and lower semi-continuous and possibly non-smooth. 
We consider a similar structure in a multi-level problem.

 \textcolor{black}{There are only a few solution approaches presented in
the literature for tri-level problems, addressing very restricted classes 
of problems, and mostly without guarantees of global optimality. For example, error-bound conditions were used by 
Senter and Dotson \cite{senter1974approximating} to assure the existence of strong convergence results for Mann iterates. 
Typically, error bounds are essential for assessing the accuracy and reliability of numerical approximations or algorithms and, providing a measure of how close the approximate solution is to the true solution, given certain assumptions or conditions. These conditions may include properties of the problem, the algorithm used, the precision of numerical calculations, and any assumptions made during the approximation process.
Very recently, Sato et al.\ \cite{Sato2021} presented a gradient-based algorithm for multilevel optimization, 
where the lower-level problems are replaced by steepest descent update equations. 
They present conditions when this reformulation asymptotically converges to the original multilevel problem. Based on our knowledge, no other solution approach can tackle the class of problems considered in this work.}

\textcolor{black}{
\subsection{Examples}
We present several concrete examples of trilevel optimization problems.
\begin{enumerate}
    \item Pursuit-evasion-intercept, or alternatively described as pursuit-evade-defend (see, e.g.~\cite{fisac2015pursuit}) is a (sequential) game wherein one player is seeking to follow and capture another in a dynamic setting, with a third tasked with intercepting the pursuer or 
    \item Bilevel optimization with robust uncertainty. Robust optimization, i.e., choosing the optimal outcome upon the worst case realization of a parameter. This can be expressed as a nested optimization problem, wherein the inner problem is a maximum over the parameter set~\cite{bental2009robust}. Any classic bilevel optimization, for instance Stackelberg games, can become trilevel when the leader makes a decision under robust uncertainty consideration.
    \item Mixture models with training and validation: consider some convex loss function on data with a regularization (e.g., LASSO), wherein the validation (for instance, a coreset) data set is considered more significant and thus an inner problem, the training set presents the middle problem, and the tuning of mixture weights of different models is the uppermost layer.
\end{enumerate}
}

\subsection{Preliminaries} 
Let $\Omega\subseteq\mathbb{R}^n$ be closed and convex and let $T$ be a mapping from $\mathbb{R}^n$ into itself. Recall that the notion of the variational inequality (VI), denoted by $VI(T,\Omega)$, is to find a vector $x^*\in\Omega$ such that 
\begin{equation}\label{eq:VI}
    VI(T,\Omega )\,\,\,\,\,\,\,\,\,\,\,\,\,\left\langle {{x^*} - T({x^*}),x - {x^*}} \right\rangle  \ge 0\,\,\,\,\,\,\,\,\,\,\forall x \in \Omega .
\end{equation}
Note that \eqref{eq:VI} is equivalent to finding the fixed point of the problem
\[\text{Find}\,\,{x^*} \in \Omega \,\,\text{such that}\,\,\,{x^*} = \textcolor{black}{{P_\Omega }T({x^*})},\]
where $P_{\Omega}$ is the metric projection of $\mathbb{R}^n$ onto $\Omega$, i.e., it maps $x\in\mathbb{R}^n$ to the unique point in $\Omega$ defined as, where throughout the paper we use the Euclidean norm,
\[{P_\Omega }(x) := \mathop {\arg \min }\limits_{y \in \Omega } \|x - y\|~~~\forall x\in\mathbb{R}^n,\]
satisfying
\[\|x - {P_\Omega }(x)\| = \inf \left\{ {\|x - y\|:{\mkern 1mu} {\mkern 1mu} {\mkern 1mu} y \in \Omega } \right\}: = d(x,\Omega ),\,\,\,\,\,\,\,{P_\Omega }(x) \in \Omega .\]
We also use the notation $\Fix(T)=\{x\in\mathbb{R}^n:~ T(x)=x\}$ for the set of fixed points of $T$. 
This set $\Fix(T)$ is closed and convex for non-expansive mappings $T$. 

We now review some well-known facts about non-expansive mappings that we shall henceforth use in the paper without reference.
\begin{itemize}
    \item Let $T:\RR^n\to\mathbb{R}^n$ be a non-expansive mapping. Then $I-T$ is monotone; that is
    \[\left\langle {x - y,(I - T)x - (I - T)y} \right\rangle  \ge 0.\]
    
    \item Let $S:\mathbb{R}^n\to\mathbb{R}^n$ be a contraction mapping with coefficient $r\in (0,1)$. Then $I-S$ called $(1-r)$-strongly monotone; that is 
    
    \[\left\langle {x - y,(I - S)x - (I - S)y} \right\rangle  \ge (1 - r)\|x - y\|^2,\,\,\,\,\,\,\,\,\,\,\forall x,y \in\mathbb{R}^n.\]
    \item Let $x\in\mathbb{R}^n$ and $z\in\Omega$ be given. Then $z=P_{\Omega}(x)$ if and only if the following inequality holds \[\left\langle {x - z,z - y} \right\rangle  \ge 0\,\,\,\,\,\,\,\forall y \in \Omega,\]
    and also if and only if \[\|x - z\|^2 + \|y - z\|^2 \le \|x - y\|^2,\,\,\,\,\,\,\,\,\,\forall y \in \Omega .\]
    \item For all $x,y\in\mathbb{R}^n$ one has \[\|{P_\Omega }(x) - {P_\Omega }(y)\|^2 \le \left\langle {{P_\Omega }(x) - {P_\Omega }(y),x - y} \right\rangle. \]
    
\end{itemize}

\section{Convergence Analysis for Trilevel Optimization Problems}

Let us consider the trilevel problem \eqref{prob:tri-levelintro}. After presenting our assumptions and preliminaries, we analyze the convergence of a  
proximal-gradient algorithm first under a variety of conditions on the step sizes (Section \ref{sec:asymptotic2}, using lemmas from Section \ref{sec:asymptotic1}).
Alternatively, one can assume a certain error-bound condition (Section \ref{sec:errorbound}).


\subsection{Assumptions and the Algorithm}

 We shall make the following standing assumptions: 
 \begin{assume}\label{s:assumptiontri}
 (i) $f_i:\mathbb{R}^n\to\mathbb{R}$ are convex and continuously differentiable with $L_{f_i}$-Lipschitz gradient, that is, 
 \[\|\nabla {f_i}(x) - \nabla {f_i}(y)\| \le {L_{{f_i}}}\|x - y\|,\,\,\,\,\forall x,y \in\mathbb{R}^n,\,\,\,\,i = 1,2.\]
 (ii) $g_i:\mathbb{R}^n\to (-\infty,+\infty]$ is proper, lower semi-continuous and convex.\\  
 (iii) the optimal solution set of the inner layers is non-empty, i.e., $X^*\ne\emptyset$ and $Y^*\ne\emptyset$.\\
 (iv) $\omega:\mathbb{R}^n\to\mathbb{R}$ is strongly convex with strong convexity parameter $\mu$.\\
 (v) $\omega$ is a continuously differentiable function so that $\nabla\omega(\cdot)$ is Lipschitz continuous with constant $L_{\omega}$.
 \end{assume}
 
 \begin{table}[t]
    \centering
    \begin{tabular}{c|c|c|c|c}
         Layer & Function & Operator & Solution Set & Optimality Condition \\ \hline
         Top & $\textcolor{black}{\omega(x)}$  & $S$ &  & $\langle x^*-S(x^*),x-x^*\rangle \ge 0,\,\forall x \in X^*$ \\ \hline
         Middle & $f_2(x)+g_2(x)$ & $T$ & $X^*\subseteq Y^*$ & $\langle x^*-T(x^*),x-x^*\rangle \ge 0,\,\forall x\in\Fix(W)=Y^*$
         \\ \hline
         Bottom & $f_1(x)+g_1(x)$ & $W$ & $Y^*$ & $x^*\in\Fix(W)$
    \end{tabular}
    \caption{The three layers of a trilevel problem and the corresponding operators. }
    \label{tab:3layers}
\end{table}

Consider three operators corresponding to the three layers of objective functions,\textcolor{black}{
\begin{equation}\label{eq:maps}
\begin{array}{l}
S(x):=S_u(x) = x-u\nabla \omega(x), \\
T(x):=T_t(x) = \prox_{t g_2} (x-t\nabla f_2(x)),\\
W(x):=W_s(x) = \prox_{s g_1} (x-s\nabla f_1(x)).\\
\end{array}
\end{equation}}
Note that each of these corresponds to a fixed point map of their respective problems. It can be easily seen that $X^*\cap\Fix(T)=\Fix(T)\cap\Fix(W)$, however, in general, we are only interested in a (specific subset) of $\Fix(W)=Y^*$ and we expect $\Fix(T)\cap\Fix(W)$ to be empty.
 
It is well known that mappings $T$ and $W$ are non-expansive and $S$ is an $r$-contraction, i.e., for any  $u\in (0,\frac{2}{{{L_\omega } + \mu }}]$, one has
\begin{align}
\label{rcontraction}
    \|S(x) - S(y)\| \le \sqrt {1 - \frac{{2u\mu {L_\omega }}}{{\mu  + {L_\omega }}}} \|x - y\|,
\end{align}
(For more details, see \cite[Theorem 2.1.12, p.66]{nesterov2003introductory}).

For any proper, lower semi-continuous and convex function $g:\RR^n\to(-\infty,+\infty]$ the Moreau proximal mapping is defined by
 \begin{equation}\label{eq:prox}
     \prox_g(x) = \mathop {\arg \min }\limits_{u\in\mathbb{R}^n} \left\{ {g(u) + \frac{1}{2}\|u - x\|^2} \right\}.
 \end{equation}

 In general, a proximal-gradient algorithm \cite{beck2014first} is based on an iterated mapping:
 \[{T_t}(x) = \prox{_{tg}}(x - t\nabla f(x)),\]
 which has the following properties:
 
 (i) $T_t$ is non-expansive for sufficiently small $t$, i.e.,
\begin{equation}\label{eq:nonexpproxg}
  \|{T_t}(x) - {T_t}(y)\| \le \|x - y\|\,\,\,\,\,\,\,\forall x,y \in\RR^n,\,\,\,\forall t \in \left(0,\frac{1}{{{L_f}}}\right].
\end{equation}
 (ii) Its fixed points are equivalent to the set of minimizers to the corresponding minimization problem, i.e.,
 \begin{equation}\label{eq:fixproxg}
  \Fix(T_t)=\mathop {\arg \min }\limits_{x \in\mathbb{R}^n} [f(x) + g(x)]~~~\forall t>0. 
 \end{equation}
 We shall denote the proximal gradient mapping as $T$ in the Algorithm, instead of $T_t$, because of property (ii).
 
 A solution $x^*$ of~\eqref{prob:tri-levelintro} satisfies the following inequalities,
\begin{equation}\label{eq:optcond}
\begin{array}{l}
\text{Find }x^*\in Y^*=\Fix(W)\\
\langle x^*-T(x^*),x-x^*\rangle \ge 0,\,\forall x\in\Fix(W)=Y^*\\
\langle x^*-S(x^*),x-x^*\rangle \ge 0,\,\forall x \in X^*.
\end{array}
\end{equation}
Note that $x\in X^*$ is equivalent to $\langle x^*-T(x^*),x-x^*\rangle \ge 0,\,\forall x\in\Fix(W)=Y^*$ so the third condition can be modified to: $\langle x^*-S(x^*),x-x^*\rangle \ge 0,$ for all $x$ satisfying this relation.
 
 For reference regarding problems ~\eqref{prob:tri-levelintro} and ~\eqref{eq:bilevelproblem}, we present\,Table ~\ref{tab:3layers}.
 
Throughout this section, we are concerned with Algorithm \ref{algorithm:trilevel},
 which is based on the proximal-gradient maps and their following combination:
\begin{equation}\label{eq:trilevelalg}
x^{k+1} = \alpha_k S(x^k)+(1-\alpha_k) \beta_k T(x^k)+(1-\alpha_k)(1-\beta_k) W(x^k).
\end{equation}

\textcolor{black}{Define the following quantities regarding the relative limit behaviors of the two parameters:
\begin{equation}\label{eq:defdel}
\delta:=\mathop {\lim \sup }\limits_{k\to\infty} \frac{{{\beta _k}}}{{{\alpha _k}}}\in[0,\infty] \text{ and }\widetilde\delta:=\mathop {\lim }\limits_{k\to\infty} \frac{{{\beta _k}}}{{{\alpha _k}}}\in [0,\infty].
\end{equation}
These quantities play a central role in analyzing the convergence of the proposed algorithms. For more details, you may see Example \ref{exam:delta}.}

The following key Assumptions will be needed throughout the paper: 
\begin{assume} \label{as:assumealpha}
$\alpha_k\to 0$ (as $k\to\infty$) and $\sum\limits_{k= 1}^{ + \infty } {{\alpha _k}} =\infty.$
\label{ass:ii}
\end{assume}

\begin{assume}\label{as:assumeabk}
There exists $K>0$ such that
$\limsup_{k\to\infty}\frac{1}{\alpha_k} \lvert \frac{1}{\beta_k}-\frac{1}{\beta_{k-1}} \rvert \le K.$
\label{ass:iv}
\end{assume}

\begin{assume}\label{as:assumelimpsup}
$\mathop {\lim \sup }\limits_{k\to\infty}\frac{{\lvert{\beta _k} - {\beta _{k - 1}}\rvert + \lvert{\alpha _k} - {\alpha _{k - 1}}\rvert}}{{{\alpha _k}{\beta _k}}} = 0.$  
\label{ass:v}
\end{assume}
\begin{algorithm}[t!]
    \caption{Proximal Gradient for Tri-level Optimization} \label{algorithm:trilevel}
 \textbf{Input:}
$t\in \left(0,\frac{1}{{{L_{{f_1}}}}}\right], s\in \left(0,\frac{1}{{{L_{f_2}}}}\right],\,r \in \,\left(0,\frac{2}{{{L_\omega } + \mu }}\right]$ and the real sequences $\alpha_k$ and $\beta_k$ satisfy the \textcolor{black}{assumptions} \\
          \textbf{Initialization:} Select an arbitrary starting point $x^0\in\mathbb{R}^n$\\
           \textbf{For} $k=1,2,...$ do 
              \[\begin{array}{l}
 {y^k}: = \prox_{t{g_2}}\left[{x^{k - 1}} - t\nabla {f_2}({x^{k - 1}})\right]  \\ 
 {z^k}: = \prox_{s{g_1}}\left[{x^{k - 1}} - s\nabla {f_1}({x^{k - 1}})\right] \\ 
 {v ^k}: = {x^{k - 1}} - r\textcolor{black}{\nabla\omega} ({x^{k - 1}})  \\ 
 {x^k} = {\alpha _{k}}v^{k} + (1 - {\alpha _{k}}){\beta _{k}}y^{k} + (1 - {\beta _{k}})(1 - {\alpha _{k}})z^{k} \\ 
 \end{array}\]
 \textbf{End For}
    \end{algorithm}
\subsection{Properties of the limit points}
 \label{sec:asymptotic1}
 
We now derive a set of results regarding the properties of limit points generated
by the sequence~\eqref{eq:trilevelalg}. 
First, we present the following powerful lemma that we shall use in the analysis below:
\begin{lem}\cite[Lemma 2.1]{xu2002iterative}
\label{lemma:L1}
Assume that $a_k$ be a sequence of non-negative real numbers such that \[{a_{k + 1}} \le (1 - {\gamma _k}){a_k} + {\delta _k},\] where $\gamma_k$ is a sequences in $(0,1)$ and $\delta_k$ is a sequence in $\mathbb{R}$, such that\\
(1) $\sum\limits_{k = 1}^\infty  {{\gamma _k} = } \infty$,\\
(2) either $\limsup\limits_{k\to\infty}\frac{\delta_k}{\gamma_k}\le 0$ or $\sum\limits_{k= 1}^\infty  {\lvert{\delta _k}\rvert < } \infty$. \\
Then $\mathop {\lim }\limits_{k\to \infty } {a_k}=0.$
\end{lem}
From now on and throughout the paper, we denote by $\{x^k\}$ the sequence generated by the algorithm \eqref{eq:trilevelalg}. The convergence of the algorithm crucially depends on the starting points $x^0\in\mathbb{R}^n$ and the parameters (step-sizes) $\alpha_k$ and $\beta_k$, which are chosen in advance.
Three different cases can be distinguished: $\delta=0, \delta>0$ and $\delta=\infty$, each associated with some other Assumptions. Initially, we are going to seek the conditions ensuring boundedness of the sequence of iterates $\{x^k\}$. Our proof techniques are similar to those that \cite{Sabach2017} used to prove their Lemma 2.

Throughout this paper, to simplify the notation, we will use $w(\{x^k\})$ to denote the set of cluster points of sequence $\{x^k\}$, i.e., 

\[w(\{x^k\}) = \left\{ {x \in\mathbb{R}^n:\,\,\,{x^{{k_i}}} \to x \textrm{ for some sub-sequence } \{ {x^{{k_i}}}\} \textrm{ of } \{ {x^k}\} } \right\},\]
and also for every $k\ge 1$, we define 
\[{Q_k}(x) = {\beta _k}T(x) + (1 - {\beta _k})W(x).\] 
It is straightforward to see that $Q_k$ is non-expansive. 

\begin{lem}\label{lem:boundedness}
Assume $\delta<\infty$ . Then $\{x^k\}$ is bounded, i.e., for every $x\in \Fix(W)$ there exists a constant $C_x$ such that $\|x^k-x\|\le C_x$ and constants $C_S$ and $C_T$ such that \[\|W({x^k}) - x\| \le {C_x},\,\,\,\|S({x^k}) - x\| \le {C_S+C_x},\,\,\,\,\,\|T({x^k}) - x\| \le {C_T+C_x}.\]
Moreover, for all $x\in\Fix(W)$ one has \[\mathop {\lim \sup }\limits_k \left(\|x^{k+1}-x\|-\|x^k-x\|\right)\le0.\]
\end{lem}
\begin{proof}
 Taking into account $\delta\in[0,+\infty)$, from Assumption \ref{ass:i} one sees that there exists $\delta_0>\delta$ and $k_0\in\mathbb{N}$ such that for every $k\ge k_0$, one has $\beta_k<\delta_0\alpha_k$. On the other hand the sequence $\{x^{k+1}\}$ can easily be rewritten as 
 \[{x^{k + 1}} = {\alpha _k}S({x^k}) + (1 - {\alpha _k}){Q_k}({x^k}).\]
 Now, for given $x\in \Fix(W)$ we obtain 
 \begin{align}
    \|{x^{k + 1}} - x\| &\nonumber= \|{\alpha _k}(S({x^k}) - S(x)) + {\alpha _k}(S(x) - x) \\
    &\nonumber+ (1 - {\alpha _k})({Q_k}({x^k}) - {Q_k}(x))+(1 - {\alpha _k})({Q_k}(x) - x)\| \\
    &\label{eq:bounded}\le {\alpha _k}r\|{x^k} - x\| + {\alpha _k}\|S(x) - x\|\\
    &\nonumber+ (1 - {\alpha _k})\|{x^k} - x\| + (1 - {\alpha _k}){\beta _k}\|T(x) - x\|\\
    &\nonumber\le(1 - (1 - r){\alpha _k})\|{x^k} - x\| +{\alpha _k}(\|S(x) - x\| + {\delta _0}\|T(x) - \;x\|)\\
    &\nonumber\le \max \left\{ {\|{x^k} - x\|,\frac{1}{{1 - r}}(\|S(x) - x\| + {\delta _0}\|T(x) - x\|)} \right\}\\
    &\nonumber\le \max \left\{ {\|{x^{k - 1}} - x\|,\frac{1}{{1 - r}}(\|S(x) - x\| + {\delta _0}\|T(x) - x\|)} \right\}\\
    &\nonumber\le ... \le \max \left\{ {\|{x^{{k_0}}} - x\|,\frac{1}{{1 - r}}(\|S(x) - x\|+ \delta_0 \|T(x) - x\|)} \right\}: = {C_x}.
 \end{align}
 There, $r$ is the coefficient of the contraction map $S$.
 Therefore, $\{x^k\}$ is bounded.
 Also, for given $x\in\Fix(W)$ from \eqref{eq:bounded} one can observe that 
 \begin{equation}\label{eq:bound2}
     \|{x^{k + 1}} - x\| \le (1-(1-r)\alpha_k)\|{x^k} - x \|+ {\alpha _k}\|S(x) - x \|+ (1 - {\alpha _k}){\beta _k}\|T(x) -x\|,
 \end{equation}
 which implies that \[\mathop {\lim \sup }\limits_k (\|x^{k+1}-x\|-\|x^k-x\|)\le0.\]
 \end{proof}
 \begin{rem}\label{rem:boundedn}
 One can see that if 
 $\Fix(W)\cap \Fix(T)\ne\emptyset$, then $\{x^k\}$ is bounded, without taking into consideration the condition $\delta\in[0,\infty].$ Indeed when $x\in\Fix(W)\cap\Fix(T)$ by \eqref{eq:bound2} we then have 
 \begin{align*}
     \|{x^{k + 1}} - x\| &\le (1-(1-r)\alpha_k)\|{x^k} - x \|+ {\alpha _k}\|S(x) - x \|\le\max\{\|x^k-x\|,\frac{\|S(x)-x\|}{1-r}\}\\
 &\le...\le\max\{\|x^{k_0}-x\|,\frac{\|S(x)-x\|}{1-r}\},
 \end{align*}
 which shows that $\{x^k\}$ is bounded.
 \end{rem}
The following simple example shows that the boundedness of $\{x^k\}$ does not necessarily hold when $\delta=\infty$.
\begin{example}
Take $X=\mathbb{R}$ and $\alpha_k=\frac{1}{k}$ and $\beta _k =\frac{1}{\sqrt{k}}$. Furthermore, let $S(x)=\frac{x}{4},T(x)=x+5, W(x)=x$. Clearly, $S$ is contraction, and $T$ and $W$ are non-expansive. It is easy to check that  $\mathop {\delta  = \lim \sup }\limits_k \frac{{{\beta _k}}}{{{\alpha _k}}} = \infty $, and  $x^k\to\infty$ from starting point $x^0=1$. 
\end{example}

The next Lemma will be useful in the sequel the proof is slightly similar to \cite[Theorem 4.1]{lu2009hybrid}
\begin{lem}\label{lem:l2}
Suppose that $\{x^k\}$ is bounded.\\\\
(a) If Assumption \ref{ass:v} holds, then $\{x^k\}$ is asymptotically regular, i.e,  
\[\mathop {\lim }\limits_{k \to \infty } \left\|{x^{k + 1}} - {x^k}\right\| = 0.\]
(b) If Assumptions \ref{ass:ii}, \ref{ass:iv} and \ref{ass:v} hold, then $\mathop {\lim }\limits_{k \to \infty } \left\|\frac{{{x^{k + 1}} - {x^k}}}{{{\beta _k}}}\right\| = 0.$\\
(c) one has $w(\{x^k\})\subseteq \Fix(W).$
\end{lem}
\begin{proof}
Since $\{x^k\}$ is bounded then there exists constant $M$ such that 
\[M \ge \mathop {\sup }\limits_{k \ge 1} \left\{ {\|S({x^{k - 1}})\|,\|T({x^{k - 1}})\|,\|W({x^{k - 1}})\|} \right\}.\]
So we have 
\begin{align*}
 \| {{Q_k}({x^k}) - {Q_{k - 1}}({x^{k - 1}})} \| &= \| {{Q_k}({x^k}) - {Q_k}({x^{k - 1}}) + {Q_k}({x^{k - 1}}) - {Q_{k - 1}}({x^{k - 1}})} \|\\
  &= \| {{Q_k}({x^k}) - {Q_k}({x^{k - 1}}) + ({\beta _k} - {\beta _{k - 1}})(T({x^{k - 1}}) - W({x^{k - 1}}))} \|\\
  &\le \| {{x^k} - {x^{k - 1}}} \| + 2M\lvert{\beta _k} - {\beta _{k - 1}}\rvert.
\end{align*}
 Now, one can write 
\begin{align}
    \|{x^{k + 1}} - {x^k}\| &\nonumber= \|({\alpha _k}S({x^k}) + (1 - {\alpha _k}){Q_k}({x^k})) - ({\alpha _{k - 1}}S({x^{k - 1}}) + (1 - {\alpha _{k - 1}}){Q_{k - 1}}({x^{k - 1}}))\|\\ 
&\nonumber=\|(1 - {\alpha _k})({Q_k}({x^k}) - {Q_{k - 1}}({x^{k - 1}})) + ({\alpha _k} - {\alpha _{k - 1}})(S({x^{k - 1}}) - {Q_{k - 1}}({x^{k - 1}}))\\ 
  &\nonumber+ {\alpha _k}(S({x^k}) - S({x^{k - 1}}))\|\\
  &\nonumber\le (1 - (1 - r){\alpha _k})\|{x^k} - {x^{k - 1}}\| + 2M\lvert{\beta _k} - {\beta _{k - 1}}\rvert + 2M\lvert{\alpha _k} - {\alpha _{k - 1}}\rvert\\
  &\label{eq:ess1}\le (1 - (1 - r){\alpha _k})\|{x^k} - {x^{k - 1}}\| +\frac{2M\lvert{\beta _k} - {\beta _{k - 1}}\rvert + 2M\lvert{\alpha _k} - {\alpha _{k - 1}}\rvert}{\beta_k}.
\end{align}
 Noticing Assumption \ref{ass:v} and
 setting $a_k=\|{x^{k+1}} - {x^k}\|$, ${\gamma _k}= (1 - r){\alpha _k}$, and \[{\delta _k}= \frac{2M\lvert{\beta _k} - {\beta _{k - 1}}\rvert + 2M\lvert{\alpha _k} - {\alpha _{k - 1}}\rvert}{\beta_k},\]
 one can apply Lemma \ref{lemma:L1} and the proof of part (a) is complete.\\
 \newline
 
 To prove part (b):
 Dividing both sides of the inequality \eqref{eq:ess1} by $\beta_k$, we obtain 
 \begin{align}
     \frac{{\|{x^{k + 1}} - {x^k}\|}}{{{\beta _k}}} &\nonumber\le ({1 - (1 - r){\alpha _k}})\left(\frac{1}{{{\beta _k}}} + \frac{1}{{{\beta _{k - 1}}}} - \frac{1}{{{\beta _{k - 1}}}}\right)\|{x^k} - {x^{k - 1}}\|\\
     &\nonumber+ 2M\left(\frac{{\lvert{\beta _k} - {\beta _{k - 1}}\rvert + \lvert{\alpha _k} - {\alpha _{k - 1}}\rvert}}{{{\beta _k}}}\right)\\
     &\nonumber\le (1 - (1 - r){\alpha _k})\left(\frac{{\|{x^k} - {x^{k - 1}}\|}}{{{\beta _{k - 1}}}} \right)+ \left\lvert\frac{1}{{{\beta _k}}} - \frac{1}{{{\beta _{k - 1}}}}\right\rvert\left(\|{x^k} - {x^{k - 1}}\|\right)\\
     &\nonumber+ 2M\left(\frac{{\left\lvert{\beta _k} - {\beta _{k - 1}}\right\rvert + \lvert{\alpha _k} - {\alpha _{k - 1}}\rvert}}{{{\beta _k}}}\right).
 \end{align}
 Using Assumptions \ref{ass:ii}, \ref{ass:iv}, \ref{ass:v} and by similar reasoning at part (a) the assertion follows from Lemma \ref{lemma:L1}.\\
 \newline 
 
 To prove part $(c):$ 
 By boundednes of $\{x^k\}$ and $\alpha_k\to0$ and $\beta_k\to0$ it is clear to see that
 \begin{align*}
   \|{x^{k + 1}} - W({x^k})\|&= \|{\alpha _k}S({x^k}) + (1 - {\alpha _k}){\beta _k}T({x^k}) + (1 - {\alpha _k})(1 - {\beta _k})W({x^k}) - W({x^k})\|\\
   &=\|{\alpha _k}S({x^k}) + (1 - {\alpha _k}){\beta _k}T({x^k}) + ({\alpha _k}{\beta _k} + {\alpha _k} + {\beta _k})W({x^k})\|\\
   &\le\alpha_k\|S({x^k})\| + (1 - {\alpha _k}){\beta _k}\|T({x^k})\| + ({\alpha _k}{\beta _k} + {\alpha _k} + {\beta _k})\|W({x^k})\|
 \end{align*}
 and we have $\|x^{k+1}-W(x^k)\|\to0$ which together with part (a) gives the conclusion of part (c).
 \end{proof}

By virtue of the prior \textcolor{black}{lemma}, we are able, in many situations, to get a unique solution for the multilevel variational inequality without additional conditions on mappings $S, T, W$. 

{\begin{table}[t]
\vspace{.2 in}
\begin{center}
\begin{tabular}{|c|c|c|c|c|}  
\hline 
Assumptions & $\delta:=\limsup\frac{\beta_k}{\alpha_k}$ & $\widetilde{\delta}:=\lim\frac{\beta_k}{\alpha_k}$& An Example& Results  \\
\hline
Ass.~\ref{ass:ii},~\ref{ass:trilevelassumption} & $\delta=0$ & ... & Ex. \ref{ex:table} (a) & Thm. \ref{thm:deltazero}\\
Ass.~\ref{ass:ii},~\ref{ass:iv},~\ref{ass:v} & $\delta\in[0,+\infty)$ & ... & Ex. \ref{ex:table} (c)& Prop. \ref{lem:deltastrictlypositive}\\
Ass.~\ref{ass:ii},~\ref{ass:iv},~\ref{ass:v},~$(A_1)$,~\textcolor{black}{$(A_2)$} & ... & $\widetilde{\delta}=\infty$ & Ex. \ref{ex:table} (d)& Thm. \ref{thm:deltainfty}\\
Ass.\ref{ass:ii},~\ref{ass:iv},~\ref{ass:v},~\ref{ass:trilevelassumption},~$(A_2)$ & $\delta\in[0,\infty)$ &... & Ex. \ref{ex:table} (f)& Prop. \ref{prop:deltainfinty}
\\
\hline
\end{tabular}
\end{center}
\caption{An overview of our results in Sections \ref{sec:asymptotic2}--
with the corresponding assumptions on $\alpha_k$ and $\beta_k$ and examples of the series.}
\label{tab:alphabetatheorems}
\end{table}

\subsection{Convergence analysis under assumptions on step-sizes $\alpha_k$ and $\beta_k$}
\label{sec:asymptotic2}
Next, we shall explore the convergence guarantees associated with different cases of $\delta$. 
We summarize these results, which depend on problem assumptions and parameter regimes, in \textcolor{black}{Table~\ref{tab:alphabetatheorems}}.
\textcolor{black}{ Let us consider the existence of a solution for the convex trilevel optimization problem \eqref{prob:tri-levelintro}. We will analyze this in multiple stages. It is worthwhile to note that the convergence behavior towards $X^*$ is made complex by the interconnection among the three layers. 
On the whole, the non-expansive operators $T$ and $W$ do not increase the distance between any two points in the iteration for $k$ large enough, and $S$ contracts the distance between points in the sequence. As step size $\alpha_k$ goes to zero, the ascendancy of the contraction mapping $S$ diminishes, and the sequence $\{x^k\}$ becomes dominated by the non-expansive mappings $T$ and $W$. The exact convergence behavior to a specific fixed point in $X^*$ will depend on additional properties of the individual operator $\Fix(W)=Y^*$ or its corresponding level $\phi_1$, such as the quadratic growth condition and linearly regular bound.
First, let us consider the consistent case, i.e., $\Fix(T)\cap\Fix(W)\ne\emptyset$, and subsequently, further cases depending on the error bound condition.}

Let us now present the key technical lemma concerning the case of $\widetilde{\delta}=\infty$, which relates to the convergence of the iteration. It establishes a connection between the set of cluster points of the sequence ${x^k}$ and the solution set of the variational inequality $VI(T,\Fix(W))$:
\begin{lem}\label{lem:l3}
Assume $\widetilde{\delta}=\infty$, together with Assumptions \ref{ass:ii}, \ref{ass:iv}, and \ref{ass:v}. Furthermore, suppose that $\{x^k\}$ is bounded. Then every cluster point of the sequence $\{x^k\}$ is in $VI(T,\Fix(W))$, i.e., 
\[w(\{x^k\}) \subseteq \left\{ {x \in \Fix(W):{\mkern 1mu} {\mkern 1mu} {\mkern 1mu} {\mkern 1mu} \left\langle {(I - T){x},y{ -x}} \right\rangle  \ge {\rm{0}}{\mkern 1mu} ,{\mkern 1mu} {\mkern 1mu} {\mkern 1mu} \forall y \in {\Fix}(W){\mkern 1mu} } \right\}.\]
\end{lem}
\begin{proof}
Let $y_k= \frac{1}{\beta_k}(x^k-x^{k+1})$ and  $x\in w(\{x^k\})$ be given. It was shown that \eqref{eq:B1} holds for every $y\in \Fix(W).$ Therefore,
\[\left\langle {{y^k},{x^k} - y} \right\rangle  \ge \frac{{{\alpha _k}}}{{{\beta _k}}}\left\langle {(I - S)({x^k}),{x^k} - y} \right\rangle  + (1 - {\alpha _k})\left\langle {(I - T)({x^k}),{x^k} -y} \right\rangle. \]
 Upon letting $k\to\infty$ in the previous inequality and utilizing the Assumption that $\widetilde{\delta}=\infty\Longleftrightarrow\mathop {\lim }\limits_k \frac{{{\alpha _k}}}{{{\beta _k}}} = 0$,  part (b) of Lemma \ref{lem:l2} ($y_k\to0$), and also boundedness of $\{x^k\}$, we are led to the following  
\[\mathop {\lim \sup }\limits_k \left\langle {(I - T)x^k,{x^k} - y} \right\rangle  \le 0,\,\,\,\,\,\,\forall y \in \Fix(W),\] 
which upon executing the limit,
\[\left\langle {(I - T)x,x - y} \right\rangle  \le 0,\,\,\,\,\,\,\forall y \in \Fix(W),\] 
and which in turn yields $x\in VI(T,\Fix(W))$.
\end{proof}

\textcolor{black}{
\begin{fact}\label{Fact.interior}
    If the interior of $X^*$ is non-empty then $X^*=\Fix(T)\cap\Fix(W)$, and as well $\{x^k\}$ is bounded.
\end{fact}
\begin{proof}
   First, we show $X^*=\Fix(T)\cap\Fix(W)$. Let there exist $x_0\in\rm{int}X^*$ and let $x\in\RR^n$ be given. Hence for sufficiently small $t\in(0,1)$, we have that $x_0+t(x-x_0)\in X^*\subset\Fix(W)$, which further implies \[{\phi _1}({x_0})={\phi _1}({x_0} + t(x - {x_0})) \le (1 - t){\phi _1}({x_0}) + t{\phi _1}({x}),\]
and so $\phi_1(x_0)\le\phi_1(x)$. This means that $x_0\in\Fix(T).$ Therefore, $\rm{int}X^*\subseteq\Fix(T)$. On the other hand, since $X^*$ is closed and convex, we therefore have \[X^*={\rm cl}({\rm int}(X^*))\subseteq\Fix(T).\]
Consequently, as we already have $X^*\cap\Fix(T)=\Fix(T)\cap\Fix(W)$, one can deduce that  $X^*=\Fix(T)\cap\Fix(W)$, which verifies the desired equality. Notably, as mentioned in Remark \ref{rem:boundedn}, it is evident that the sequence ${x^k}$ is bounded.
\end{proof}}

\textcolor{black}{
\begin{assume}\label{ass:trilevelassumption}
(Quadratic growth condition)
Suppose now that $\phi_1$ grows quadratically (globally) away
from a part of its minimizing set $Y^*=\Fix(W)$, i.e.,  $X^*$, meaning there is a real number  $\mu>0$ such that
\begin{align}\label{QGC}
    \phi_1(x)\ge\phi_1^*+\frac{\mu}{2}{\rm dist}^2(x,X^*)~~~\forall
    x\in\Omega_1\backslash\Fix(W)
\end{align}
where $\Omega_1=\mathcal{B}(0,C_{x_0})$ for given $x_0\in\Fix(W)$ and $\phi_1^*$ represents the optimal value of  $\phi_1$.
\end{assume}
The quadratic growth condition can be interpreted as a notion of sharpness assumption on the function $\phi_1$, which describes functions that exhibit at least the behavior of ${\rm dist}(x, X^*)$ .\\
Originally introduced to establish the convergence of trajectories for the gradient flow of analytic functions, Bolte et al. proposed an extension to non-smooth functions in their work published in \cite{bolte2007lojasiewicz}.\\
As a simple example, let us assume that $\phi_1(x,y)=0$ for $(x,y)\in[-1,1]\times[-1,1]$ and for otherwise $\phi_1(x,y)\ge2$
and $\Omega_1=\mathcal{B}_2(0,0)$, we get $\Fix(W)=[-1,1]\times[-1,1]$ now, considering $X^*=[-\frac{1}{2},\frac{1}{2}]\times[-\frac{1}{2},\frac{1}{2}]$, we will observe, through a straightforward investigation, that \eqref{QGC} is verified.  }  
\textcolor{black}{
\begin{thm} \label{thm:deltazero}
    Let Assumption \ref{ass:trilevelassumption} hold, and $\delta=0$. Then $\{x^k\}$ converges to some $x^*\in X^*$ such that 
\[\,\,\left\langle {{x^*} - S({x^*}),x - {x^*}} \right\rangle  \ge 0\,\,\,\,\,\,\,\,\,\,\forall x \in {X^*}.\]
\end{thm}
\begin{proof}
     Strong convexity of $\omega$, and contractivity of the operator $S$, together implies that there is unique $x^*\in X^*$ such that $x^*=P_{X^*}S x^*$ and $x^*\in VI(S,X^*)$, i.e.,  
\begin{align}\label{trilevel0}
    \left\langle {{x^*} - S({x^*}),x - {x^*}} \right\rangle  \ge 0,\,\,\,\,\,\,\,\forall x \in {X^*}.
\end{align}
Since the sequence $\{x^{k}\}$ is bounded, one sees that $x^k\in\Omega_1$.  Furthermore, utilizing assumption \ref{ass:trilevelassumption}, one may be readily verified that $w(x^k)\subset X^*$. Moreover, one can extract a convergent sub-sequence $\{x^{k_i}\}$ of $\{x^{k+1}\}$ or any sub-sequence thereof to $x^{'}\in X^*$, which holds by Lemma~\ref{lem:l2}, part c, and \eqref{trilevel0} so that
\begin{align}\label{eq:A2}
\mathop {\lim \sup }\limits_k \left\langle {S({x^*}) - {x^*},{x^{k+1}} - {x^*}} \right\rangle  &\nonumber= \mathop {\lim }\limits_i \left\langle {S({x^*}) - {x^*},{x^{{k_i}}} - {x^*}} \right\rangle\\
= \left\langle {S({x^*}) - {x^*},x^{'} - {x^*}} \right\rangle  \le 0.
\end{align}
Next, we show $x^k\to x^*$. Let the sequences $c_k$ and $d_k$ are defined as
\begin{align*}
    {c_k}:&= {\alpha _k}(S({x^k}) - S({x^*})) + (1 - {\alpha _k}){\beta _k}(T(x^k) - T(x^*))\\
    &+ (1 - {\alpha _k})(1 - {\beta _k})(W(x^k) - W({x^*})), \\
 {d_k}: &= {\alpha _k}(S({x^*}) - {x^*}) + (1 - {\alpha _k}){\beta _k}(T(x^*)  - x^*). 
\end{align*}
From above it is immediate that $c_k+d_k=x^{k+1}-x^*$, and $\|{c_k}\| \le (1 - (1 - r){\alpha _k})\|{x^k} - x^*\|$. 
By a simple calculation, one has  
\[\|{c_k} + {d_k}\|^2 \le \|{c_k}\|^2 + 2\left\langle {{d_k},{c_k} + {d_k}} \right\rangle,\]
and finally by plugging $c_k$ and $d_k$ in the previous inequality follows that
\begin{align} \label{eq:A5}
   \|{x^{k + 1}} - x^*\|^2&\nonumber\le(1 - (1 - r){\alpha _k})\|{x^k} - x^*\|^2 + 2{\alpha _k}\left\langle {S(x^*) - x^*,{x^{k + 1}} - x^*} \right\rangle\\
   &+2(1 - {\alpha _k}){\beta _k}\left\langle {T(x^*) - x^*,{x^{k + 1}} - x^*} \right\rangle.
\end{align}
Now, setting
\begin{align}\label{eq:A6}
\left\{ \begin{array}{l}
 {a_k}= \|{x^k} - {x^*}\|^2, \\ 
 {\gamma _k}= (1 - r){\alpha _k}, \\ 
 {\delta _k}= 2{\alpha _k}\left\langle {S({x^*}) - {x^*},{x^{k + 1}} - {x^*}} \right\rangle  + 2(1 - {\alpha _k}){\beta _k}\left\langle {T({x^*}) - {x^*},{x^{k + 1}} - {x^*}} \right\rangle.  \\ 
 \end{array}\right.
 \end{align}
One has that
\[a_{k+1}\le (1-\gamma_k)a_k+\delta_k.\]
Also, using the boundedness of $\{x^k\}$ together with $\delta=0$ we can conclude that $\mathop {\lim \sup }\limits_k \frac{{{\delta _k}}}{{{\gamma _k}}} \le 0.$ Indeed, taking into account \eqref{eq:A2} and $\delta=0$ gives
 \begin{align*}
   \mathop {\limsup }\limits_k \frac{{{\delta _k}}}{{{\gamma _k}}} &=\mathop{\limsup}_{k}\left[ \frac{2{\alpha _k}\left\langle {S({x^*}) - {x^*},{x^{k + 1}} - {x^*}} \right\rangle  + 2(1 - {\alpha _k}){\beta _k}\left\langle {T({x^*}) - {x^*},{x^{k + 1}} - {x^*}} \right\rangle}{(1 - r){\alpha _k}}\right]\\
   &\le\frac{2}{{1 - r}}\mathop {\lim \sup }\limits_k \left\langle {S({x^*}) - {x^*},{x^{k + 1}} - {x^*}} \right\rangle\\
   &+\frac{2}{{1 - r}}\mathop {\lim \sup }\limits_k (1 - {\alpha _k})\frac{{{\beta _k}}}{{{\alpha _k}}}\left\langle {T({x^*}) - {x^*},{x^{k + 1}} - {x^*}} \right\rangle\le 0. \\
 \end{align*}
The desired assertion now follows from Lemma \ref{lemma:L1}.
\end{proof}}
\textcolor{black}{
\begin{thm}
 Let Assumption \ref{ass:trilevelassumption} hold, and $\delta=\infty$. Moreover, assume that $\{x^k\}$ is bounded. Then $\{x^k\}$ converges to some $x^*\in X^*$ such that 
\[\,\,\left\langle {{x^*} - S({x^*}),x - {x^*}} \right\rangle  \ge 0\,\,\,\,\,\,\,\,\,\,\forall x \in {X^*},
\] 
 i.e, $ \min_{x\in X^*}\omega(x)=\omega(x^*)$.
\end{thm}
\begin{proof}
 As before, there is a unique $x^*\in X^*$ fixed point of the contraction map $P_{X^*}S$, i.e., $x^*=P_{X^*}S x^*$. Therefore $x^*\in VI(S,X^*)$ 
and 
\[\left\langle {{x^*} - S({x^*}), x - {x^*}} \right\rangle  \ge 0,\,\,\,\,\,\,\,\forall x \in {X^*}\] as a same method, due to the boundedness of $x^k$ we may get a subsequence $\{x^{k_i}\}$ converges to $x^{'}\in X^*$ such that 
\begin{align*}
\mathop {\lim \sup }\limits_k \left\langle {S({x^*}) - {x^*},{x^{k+1}} - {x^*}} \right\rangle  &\nonumber= \mathop {\lim }\limits_i \left\langle {S({x^*}) - {x^*},{x^{{k_i}}} - {x^*}} \right\rangle\\
= \left\langle {S({x^*}) - {x^*},x^{'} - {x^*}} \right\rangle  \le 0,
\end{align*}
and also there is subsequence $x^{k_j}$ converges to $x^{''}\in X^*$ such that
\begin{align}\label{eq:T}
\mathop {\lim \sup }\limits_k \left\langle {T({x^*}) - {x^*},{x^{k+1}} - {x^*}} \right\rangle  &\nonumber= \mathop {\lim }\limits_i \left\langle {T({x^*}) - {x^*},{x^{{k_i}}} - {x^*}} \right\rangle\\
= \left\langle {T({x^*}) - {x^*},x^{''} - {x^*}} \right\rangle\le0 .
\end{align}
We show the last inequality. Using Lemma \ref{lem:l3}, one derives that $x^{''}\in VI(I-T,\Fix(W))$, i.e., 
\[\left\langle {T(x^{''}) - x^{''},x - x^{''}} \right\rangle\le0~~\forall x\in\Fix(W)\]
Taking $x=x^*\in X^*\subseteq\Fix(W)$, gives $\left\langle {T(x^{''}) - x^{''},x^* - x^{''}} \right\rangle\le0$. Using
monotonicity of $I-T$ yields that $\langle T(x^*)-x^*,x^{''}-x^*\rangle\le0$ and this follows \eqref{eq:T}. The rest of the proof follows from \eqref{eq:A5} and \eqref{eq:A6}, and Lemma \ref{lemma:L1}.  
\end{proof}}
 

As another application of Theorem \ref{thm:deltazero}, one may point to Theorem 6.1 of \cite{xu2002iterative} for solving the following quadratic minimization problem:
\begin{equation}\label{eq:quadratic}
    \mathop {\min }\limits_{x \in K} [\omega(x):=\frac{\mu }{2}\left\langle {Ax,x} \right\rangle  + \frac{1}{2}\|x - u\|^2 - \left\langle {x,b} \right\rangle],
\end{equation}
where $K$ is a nonempty closed convex  and $\mu\ge 0$ is a real number, $u,b\in\RR^n$ and $A$ is a bounded linear operator which is positive ($\left\langle {Ax,x} \right\rangle  \ge 0$ for all $x\in\RR^n$). Set $S(x):=x-r\nabla\omega(x)$ and $T(x):=\prox_{\delta_K}(x)=P_K(x)$. Then the sequence $\{x^k\}$ generated by $x^{k+1}=\alpha_kS(x^k)+(1-\alpha_k)T(x^k)$ converges to the unique solution $x^*$ of problem \eqref{eq:quadratic} under the mild assumption $\alpha_k\to0$ and $\sum\limits_k {{\alpha _k}}  = \infty$. We drop the assumption that $\mathop {\lim }\limits_k \frac{{{\alpha _{k + 1}}}}{{{\alpha _k}}} = 1$. Notice that when we take $K=\RR^n$, then problem \eqref{eq:quadratic} reduces to a classical convex quadratic optimization problem, in which  case $x^k\to\prox_f(u)$ where
$f(u) = \frac{1}{2}\left\langle {Au,u} \right\rangle  + \left\langle {u,b} \right\rangle$ and $\prox_f(u)=(A+I)^{-1}(u-b)$.

\begin{rem}\label{rem:deltazero }
Knowing relation \eqref{eq:A5} and Assumption \ref{ass:ii}, we find out that the two following conditions together imply the convergence of the sequence $\{x^k\}$: 
\begin{equation}
    \mathop {\lim \sup }\limits_k \left\langle {(S - I){x^*},{x^k} - {x^*}} \right\rangle  \le 0,
\end{equation}
and
\begin{equation}\label{C2}
\mathop {\lim \sup }\limits_k \frac{{{\beta _k}}}{{{\alpha _k}}}\left\langle {(T - I){x^*},{x^k} - {x^*}} \right\rangle  \le 0.
\end{equation}
Thanks to the Assumptions of Theorem \ref{thm:deltazero}, $x^*$ solves $VI(S,X^*)$. This is due to the fact that $\delta=0$, which means $\beta_k\to 0$ faster than $\alpha_k\to 0$. Afterwards, the term $\alpha_k S(x^k)$ dominates, while the
term $\beta_k T(x^k)$ becomes negligible. When $\delta=\infty$, it is difficult to confirm the verification of condition \eqref{C2} without assuming bounded linear regularity to control the growth of $\|x-T(x)\|$.
\end{rem}
Up to now, we have shown that the sequence $\{x^k\}$ is bounded and convergent, provided that $\delta=0$. A natural question is to ask  whether the sequence $\{x^k\}$ is convergent  when $\delta$ is non-zero.
The following proposition guarantees, under the assumption $\delta\in[0,+\infty)$, that there is a particular variational inequality that is satisfied for any limit point of the sequence generated by the Algorithm.
 \begin{prop}\label{lem:deltastrictlypositive}
Assume $\delta<\infty$, together with Assumptions \ref{ass:ii}, \ref{ass:iv} and \ref{ass:v}. Then sequence $\{x^k\}$ converges to the unique solution of the variational inequality
\begin{equation}\label{A4.}
   \,\,\,\exists \widetilde{x} \in \Fix(W)\,\,\,\,\,\,\,\left\langle {(I - S)\widetilde{x} + \delta (I - T)\widetilde{x},x - \widetilde{x}} \right\rangle  \ge 0,{\mkern 1mu} {\mkern 1mu} {\mkern 1mu} {\mkern 1mu} {\mkern 1mu} {\mkern 1mu} \forall x \in \Fix(W).
\end{equation}
\end{prop} 
 
 \begin{proof}
Set $y_k=\frac{1}{\beta_k}(x^{k}-x^{k+1})$.  From part (b) of  Lemma \ref{lem:l2} we have $y_k\to0$ as $k\to\infty$. By the definition of iteration \eqref{eq:trilevelalg} and  monotonicity of $I-W$ for all $y\in \Fix(W)$, one sees easily that 
\begin{align}
 \left\langle {{y^k},{x^k} - y} \right\rangle &\nonumber=
 \frac{{{\alpha _k}}}{{{\beta _k}}}\left\langle {(I - S)x^k,{x^k} - y} \right\rangle  + (1 - {\alpha _k})\left\langle {(I - T)x^k,{x^k} - y} \right\rangle\\
 &\nonumber+\frac{{(1 - {\alpha _k})(1 - {\beta _k})}}{{{\beta _k}}}\left\langle {(I - W)x^k - (I - W)y,{x^k} - y} \right\rangle\\
 &\label{eq:B1}\ge\frac{{{\alpha _k}}}{{{\beta _k}}}\left\langle {(I - S)x^k,{x^k} - y} \right\rangle  + (1 - {\alpha _k})\left\langle {(I - T)x^k,x^k-y} \right\rangle,
\end{align}
which implies that 
\begin{equation}\label{eq:last}
\frac{{{\beta _k}}}{{{\alpha _k}}}\left\langle {{y^k},{x^k} - y} \right\rangle  \ge \left\langle {(I - S)x^k,{x^k} - y} \right\rangle  + \frac{{{\beta _k}(1 - {\alpha _k})}}{{{\alpha _k}}}\left\langle {(I - T)x^k,{x^k} - y} \right\rangle.
\end{equation}
Now, for given $x_1,x_2\in w(\{x^k\})$, there exist sub-sequences $\{x^{k_i}\}$ and $\{x^{k_j}\}$ of $\{x^k\}$, such that $x^{k_i}\to x_1$ and $x^{k_j}\to x_2$. On taking the limsup of \eqref{eq:last} and using the fact that $y_k\to 0$ and $\mathop {\lim \sup }\limits_k \frac{\beta_k}{\alpha_k}= \delta\in [0,\infty)$, we deduce that,
\begin{equation}\label{eq:rearrange}
 \left\{ \begin{array}{l}
 \left\langle {\left. {\delta (I - T)x_1 + (I - S)x_1,x_1 - y} \right\rangle } \right. \le 0,\, \\ 
 \left\langle {\left. {\delta (I - T)x_2 + (I - S)x_2,x_2 - y} \right\rangle } \right. \le 0,\, \\ 
 \end{array} \right.\,\,\,\,\,\,\,\forall y \in \Fix(W).
\end{equation}
Rearranging \eqref{eq:rearrange} by substituting $y=x_1$ and $y=x_2$ shows that
\begin{align} 
    \left\langle {(I - S)x_1,x_1 - x_2} \right\rangle  &\label{eq:B2}\le  - \delta \left\langle {(I - T)x_1,x_1 - x_2} \right\rangle, \\
    - \left\langle {(I - S){x_2},x_1 - x_2} \right\rangle  &\label{eq:B3} \le \delta \left\langle {(I - T)x_2,x_1 - x_2} \right\rangle. 
\end{align}
On the other hand, since $I-S$ is $(1-r)$-strongly monotone and $I-T$ is monotone by adding up inequalities \eqref{eq:B2} and \eqref{eq:B3}, one obtains that 
 \begin{align*}
    (1 - r)\|x_1 - x_2\|^2 &\le \left\langle {(I - S)x_1 - (I - S)x_2,x_1 - x_2} \right\rangle\\
    &\le  - \delta \left\langle {(I - T)x_1 - (I - T)x_2,x_1 - x_2} \right\rangle\\
    & \le 0.
 \end{align*}
 So, $x_1=x_2$. This shows that $\{x^k\}$ converges. (Here, we have used the fact that the sequence $\{x^k\}$ converges if and only if every sub-sequence of $\{x^k\}$ contains a convergent sub-sequence.) Setting $\widetilde{x}:=\mathop {\lim }\limits_{k \to \infty } {x^k}$, we then see from \eqref{eq:rearrange} that
 \[ \left\langle {(I-S)\widetilde{x}+\delta(I-T)\widetilde{x},x-\widetilde{x}} \right\rangle\ge0 ,\,\,\,\,\,\,\forall x \in\Fix(W).\]
This completes the proof.
\end{proof}
\textcolor{black}{
\begin{cor}
    For each operator $P\in[I-S,I-T]$ one has $VI(P,\Fix(W))\ne\emptyset$
    where $c\in[a,b]$ means that there is $t\in[0,1]$ such that $c=ta+(1-t)b$.
\end{cor}}

\subsection{Convergence analysis under an error-bound condition}
\label{sec:errorbound}

Here, we introduce an error-bound condition that facilitates additional convergence guarantees. 
Let us denote the closed ball of radius $\rho$ centred at $0$ by $\mathcal{B}(0,\rho)$.
 
 \textcolor{black}{\begin{defin} {\bf (Error bound condition), \cite{borwein2017convergence}}
Let $W:X\to X$ be such that $\Fix(W)\ne\emptyset$. We say that $W$ is boundedly  linearly regular if 
\[\forall \rho  > 0\,\left( {\exists \theta > 0} \right)\,\left( {\forall x \in\mathcal{B}(0;\rho )} \right)\,\,\,d(x,\Fix(W)) \le \theta\|x - W(x)\|.\]
note that in general $\theta$ depends on $\rho$, which we sometimes indicate by writing $\theta = \theta(\rho)$. 
\end{defin}}
\textcolor{black}{
The notion of a bounded linear regularity is a valuable property in optimization and variational analysis. It ensures that a function behaves well near its critical points, and has been used in \cite{bauschke2015linear} to analyze linear
convergence of algorithms involving nonexpansive mappings. An exemplary and practically significant illustration of an objective that is non-quasi-strongly convex yet satisfies the quadratic growth condition is the LASSO problem:
\begin{align}
    \min_{x\in\RR^n}[\frac{1}{2}\|Qx-b\|^2+\lambda\|x\|_1]
\end{align}
when the operator $Q$ has a nontrivial kernel. Further classes of functions that possess a regular error-bound property include the following:
\begin{example}
\[\begin{array}{l}
 X=\mathbb{R}^n,~~\|.\|=\|.\|_2 ~f(x) = \|Ax - b\|^2_2,\,\,\,\,g(x) = \|x\|_1\,\,\,\,\ \\ 
 X=\mathbb{R}^n,~~\|.\|=\|.\|_2 ~f(x) = \|Ax - b\|^2_2,\,\,\,\,\,g(x) = \left\{ \begin{array}{l}
 0\,\,\,\,\,\,\,\,\,\,\,x \in \mathcal{B}_2^{\frac{n}{2}} \\ 
  + \infty \,\,\,\,\,\,{\rm else} \\ 
 \end{array} \right. \\ 
 \|.\| = \|.\|_F, X=\mathbb{R}^{p\times n},\,\,\,\,\,\,\,\,\,f(x) =\|Ax - b\|^2_2,\,\,\,\,\,\,\,\,g(x) = \|x\|{_{\rm nuc}} \\ 
 \end{array}\]
\end{example}
}

\begin{prop}\label{prop:reqularity}
For all $t\in\left(0,\frac{1}{L_{f_2}}\right]$ and $x\in\dom(\partial\phi_2)$ one has \[\|x - {T_t}(x)\| \le td(0,\partial {\varphi _2}(x)).\]
\end{prop}
\begin{proof}
For every $x\in\dom(\partial\phi_2)$ we have $\partial {\varphi _2}(x) = \nabla {f_2}(x) + \partial {g_2}(x)$. For given $t\in(0,\frac{1}{L_{f_2}}]$ and $z\in\partial\phi_2(x)$ one has 
\[(x - t\nabla {f_2}(x)) + tz \in x + t\partial {g_2}(x) = (I + t\partial {g_2})(x),\]
 we must then have \[{(I + t\partial {g_2})^{ - 1}}((x - t\nabla {f_2}(x)) + tz) = x,\]
or equivalently, $T_t(x+tz)=\prox_{tg_2}{(x+tz - t\nabla {f_2}(x))}= x.$ Since the proximal mapping is non-expansive, we deduce that 
\[\|x - {T_t}(x)\| = \|{T_t}(x + tz) - {T_t}(x)\| \le t\|z\|,\,\,\,\,\,\,\,\,\,\,\forall z \in \partial {\varphi _2}(x).\]
Letting $z$ be the minimal norm element of $\partial {\varphi _2}(x)$, we derived the claimed inequality $\|x - {T_t}(x)\| \le td(0,\partial {\varphi _1}(x))$.
\end{proof}
We shall now study cases wherein $\delta$  is not finite. From now on, we use $\Omega:=VI(T,\Fix(W))$ and assume it is non-empty. 

\begin{thm}
\label{thm:deltainfty}
Assume $\widetilde{\delta}=\infty$, together with Assumptions \ref{ass:ii}, \ref{ass:iv}, and \ref{ass:v}. Assume also that $\{x^k\}$ is bounded. Moreover, if the following assumptions hold \\
\textcolor{black}{$(A_2)$} $W$ is boundedly linearly regular,\\
\textcolor{black}{$(A_3)$} $\mathop {\lim \sup }\limits_k \frac{{{\beta _k}^2}}{{{\alpha _k}}}=0$,\\
then the sequence $\{x^k\}$ converges to $x^*$, the unique solution of
\begin{equation}
    \begin{array}{l}
 \left\langle {{x^*} - S({x^*}),x - {x^*}} \right\rangle  \ge 0,\,\,\,\,\,\,\,\forall x \in \Omega. \\ 
  \end{array}
\end{equation}
 Furthermore, this implies that $x^*$ minimizes $\omega$ over $\Omega$, i.e., $$\mathop {\min }\limits_{x \in \Omega} \omega (x) = \omega ({x^*}).$$ 
\end{thm}

\begin{proof}
Since $\Omega$ is closed and convex and $S$ is a contraction, there exists  $x^*\in\Omega$, which is a unique fixed point of the projection map $P_{\Omega}S(x^*)=x^*$, i.e.,
 \begin{equation}\label{eq:BVI}
     \left\langle {{x^*} - S({x^*}),x - {x^*}} \right\rangle  \ge 0,\,\,\,\,\,\,\,\forall x \in \Omega .
 \end{equation}
To deduce $x^k\to x^*$, we first note that  $x^*\in \Omega:=VI(T,\Fix(W))$, which implies that 
\[\left\langle {(T - I){x^*},x - {x^*}} \right\rangle\le0 \,\,\,\,\,\,\,\,\,\,\forall x \in {\Fix}(W),\]
and since $P_{\Fix(W)}(x^k)\in \Fix(W)$, one gets
 \begin{equation}\label{eq:PVI}
     \,\,\,\,\left\langle {(T - I){x^*},{P_{\Fix(W)}}({x^{k + 1}}) - {x^*}} \right\rangle\le 0.
 \end{equation}
 On the other hand, since $\{x^k\}$ is bounded, there exists $M>0$ and $k_0\ge 0$  such that for all $k\ge k_0$, one has $x^k\in\mathcal{B}(0,M)$. So, by applying Assumption $(A_1)$, one can easily observe that there exists $\theta>0$ such that
 \begin{equation}\label{eq:dVI}
 d({x^k},\Fix(W)) \le \theta \|{x^k} - W({x^k})\|.
 \end{equation}
 Employing \eqref{eq:PVI} and \eqref{eq:dVI}, one has
\begin{align*}
    \left\langle {(T - I){x^*},{x^{k + 1}} - {x^*}} \right\rangle \,\,\, &=\left\langle {(T - I){x^*},{x^{k + 1}} - {P_{\Fix(W)}}({x^{k + 1}}) + {P_{\Fix(W)}}({x^{k + 1}}) - {x^*}} \right\rangle\\
    &=\left\langle {(T - I){x^*},{x^{k + 1}} - {P_{\Fix(W)}}({x^{k + 1}})} \right\rangle \\
    &+ \left\langle {(T - I){x^*},{P_{\Fix(W)}}({x^{k + 1}}) - {x^*}} \right\rangle\\ 
    &\le \left\langle {(T - I){x^*},{x^{k + 1}} - {P_{\Fix(W)}}({x^{k + 1}})} \right\rangle\\
    &\le \|T({x^*}) - {x^*}\|\,{\mkern 1mu} \|{x^{k + 1}} - {P_{\Fix(W)}}({x^{k + 1}})\|\\
    &= \|T({x^*}) - {x^*}\|{\mkern 1mu} d\left( {{x^{k + 1}},\Fix(W)} \right)\\
    &\le \theta \|T({x^*}) - {x^*}\|{\mkern 1mu} \|{x^{k + 1}} - W(x^{k+1})\|.
\end{align*}

Hence
\begin{equation}\label{eq:C1}
     \left\langle {(T - I){x^*},{x^{k + 1}} - {x^*}} \right\rangle\le \theta \|T({x^*}) - {x^*}\|{\mkern 1mu} \|{x^{k + 1}} - W(x^{k+1})\|.
\end{equation}
Now, since $\{x^k\}$ is bounded, one can find a constant $C> 0$ so that

\[C \ge \mathop {\sup }\limits_{k \ge 1} \left\{ {\|S({x^k})\|,\|T({x^k})\|,\|W({x^k})\|} \right\},\]
and we will then have 
 \begin{align*}
     \|{x^{k + 1}} - W{x^{k + 1}}\| &\le\|x^{k+1}-W(x^k)\|+\|W(x^{k+1})-W(x^k)\|\\ 
     &\le \|{x^{k + 1}} - W{x^k}\| + \|{x^{k + 1}} - {x^k}\|\\
     &
  \le {\alpha _k}\|S({x^k})\| + {\beta _k}\|T({x^k})\| + ({\alpha _k} + {\beta _k} + {\alpha _k}{\beta _k})\|W({x^k})\| + \|{x^{k + 1}} - {x^k}\|\\ 
  &\le (2{\alpha _k} + 2{\beta _k} + {\alpha _k}{\beta _k})C + \|{x^{k + 1}} - {x^k}\|.
 \end{align*}
 Therefore, by combining the previous inequality and  \eqref{eq:C1}, we get:
 
 \begin{align}
 \left\langle {(T - I){x^*},{x^{k + 1}} - {x^*}} \right\rangle &\nonumber\le \theta \|T({x^*}) - {x^*}\|\left[ {(2{\alpha _k} + 2{\beta _k} + {\alpha _k}{\beta _k})C + \|{x^{k + 1}} - {x^k}\|} \right].\\ 
  &\label{eq:C2}
 \end{align}
 Now, multiplication \eqref{eq:C2} with $\frac{\beta_k}{\alpha_k}$ yields 
 \begin{align}
 \frac{{{\beta _k}}}{{{\alpha _k}}}\left\langle {(T - I){x^*},{x^{k + 1}} - {x^*}} \right\rangle &\label{eq:holder}\le \theta \|T({x^*}) - {x^*}\|\left[ {\left(2{\beta _k} + 2\frac{{\beta _k^2}}{{{\alpha _k}}} + {\beta _k^2}\right)C + \frac{{{\beta _k^2}}}{{{\alpha _k}}}.\frac{{\|{x^{k + 1}} - {x^k}\|}}{{{\beta _k}}}} \right]. 
 \end{align}
  Using \eqref{eq:holder}, Assumption $(A_4)$, and part (b) of Lemma \ref{lem:l2}, we will observe that 
 \begin{equation}\label{eq:C3}
    \mathop {\lim \sup }\limits_k \frac{{{\beta _k}}}{{{\alpha _k}}}\left\langle {(T - I){x^*},{x^{k + 1}} - {x^*}} \right\rangle  \le 0.\,
 \end{equation}
 Moreover by Lemma \ref{lem:l3}, we have $w(\{x^k\}
 )\subseteq \Omega$. Now, since $\{x^{k+1}\}$ is bounded there exists a convergent sub-sequence $\{x^{k_i}\}$ of $\{x^{k+1}\}$ to $x^{'}\in\Omega$. From \eqref{eq:BVI}, it can be seen that
 \begin{align}
   \mathop {\lim \sup }\limits_k \left\langle {(S - I){x^*},{x^{k + 1}} - {x^*}} \right\rangle & \nonumber=\mathop {\lim }\limits_k \left\langle {(S - I){x^*},{x^{{k_i}}} - {x^*}} \right\rangle \\
   &\label{eq:C4}=\left\langle {(S - I){x^*},x^{'} - {x^*}} \right\rangle  \le 0. 
 \end{align}
 Recall that we still have inequality \eqref{eq:A5}. By a similar argument as in Remark \ref{rem:deltazero }, from \eqref{eq:C3} and \eqref{eq:C4} and in view of Lemma \ref{lemma:L1}, we see that $x^k\to x^*$ and the proof is complete.
\end{proof}
 
 Notice that by Lemma \ref{lem:l3}, we know that when $\widetilde{\delta}=\infty$, then $w(\{x^k\})\subseteq\Omega$. The following example shows that this is not a necessary condition.
\begin{example}
Take the choices 
\[S(x) = \frac{x}{4},\,\,\,\,\,T(x) = \left\{ \begin{array}{lr}
 x & x \in [ - 1,1], \\ 
 1 & x \ge 1, \\ 
  - 1 & x \le  - 1, \\ 
 \end{array} \right.,\,\,\,\,\,\,\,\,\,\,\,W(x) = \left\{ \begin{array}{lr}
 x \,\,\,\,\,& x \in [0,2], \\ 
 2 & x \ge 2, \\ 
 0 & x \le 0. \\ 
 \end{array} \right.\]
 Also, consider $\alpha_k=\frac{1}{k}$ and $\beta_k=\frac{1}{k^2}$. It can be seen that $\delta=0$ and $\Omega=VI(T,\Fix(W))=[0,1]$ and also $w(\{x^k\})=\{0\}\subset\Omega$.
\end{example}

The following result asserts the existence of a limit point satisfying a variational inequality under a mild assumption related to the preceding theorem without any condition on $\delta$. 
 
 \begin{prop}\label{prop:deltainfinty}
Assume $\delta<\infty$, together with Assumptions \ref{ass:ii}, \ref{ass:iv}, \ref{ass:v}, and $(A_2)$. Moreover, assume that Assumption \ref{ass:trilevelassumption} holds for $\Omega$ in replace of $X^*$. Then the sequence $\{x^k\}$ converges to  $x^*$, which is the unique solution of
\begin{equation}
    \begin{array}{l}
 \left\langle {{x^*} - S({x^*}),x - {x^*}} \right\rangle  \ge 0,\,\,\,\,\,\,\,\forall x \in \Omega. \\ 
  \end{array}
\end{equation} 
\end{prop}
 
 \begin{proof}
This is immediate from Theorem \eqref{thm:deltainfty}.
\end{proof}
The following fact provides the limit of distance between $\{x^k\}$ and $\Fix(W)$ and $X^*$, respectively.
\begin{fact}\label{fact:distance}
 Let Assumption \ref{ass:v} hold and  $h_k=d(x^k,X^*)$. Also suppose that $\{x^k\}$ is bounded. Then the following assertion holds\\
 a) If $h_k\ne0$ then $\mathop {\lim }\limits_{k \to \infty } \frac{{{{(h_{k + 1}^2 - h_k^2)}^ + }}}{{h_{k+1}}} = 0{\mkern 1mu}$.\\
 b) $\mathop {\lim }\limits_{k \to \infty } d({x^k},\Fix(W)) = 0.$
\end{fact}
 
 \begin{proof}
We just prove the first assertion. (The second is straightforward from the boundedness of $\{x^k\}$.) The proof relies on the study of the sequence $\{h_k\}$.
Since $P_{X^*}$ is the projection operator onto the convex set $X^*$, we have 
\begin{align*}
\frac{1}{2}h_k^2 &= \frac{1}{2}\|{x^k} - {P_{X^*}}({x^k})\|^2\\
 &= \frac{1}{2}\|({x^k} - {P_{X^*}}({x^k})) - ({x^{k + 1}} - {P_{X^*}}({x^{k + 1}})) + ({x^{k + 1}} - {P_{X^*}}({x^{k + 1}}))\|^2\\
 &= \,\frac{1}{2}\,\|({x^k} - {P_{X^*}}({x^k})) - ({x^{k + 1}} - {P_{X^*}}({x^{k + 1}}))\|^2 + \frac{1}{2}\|{x^{k + 1}} - {P_{X^*}}({x^{k + 1}})\|^2\\
 &+\left\langle {\left. {({x^k} - {P_{X^*}}({x^k})) - ({x^{k + 1}} - {P_{X^*}}({x^{k + 1}})),{x^{k + 1}} - {P_{X^*}}({x^{k + 1}})} \right\rangle } \right.\\
 &\ge\frac{1}{2}h_{k+1}^2 + \left\langle {\left. {{x^k} - {x^{k + 1}},{x^{k + 1}} - {P_{X^*}}({x^{k + 1}})} \right\rangle } \right.\\
 &+\left\langle {\left. {{P_{X^*}}({x^{k + 1}}) - {P_{X^*}}({x^k}),{x^{k + 1}} - {P_{X^*}}({x^{k + 1}})} \right\rangle } .\right.
 \end{align*}
 Now, $P_{X^*}(x^k)\in X^*$ and consequently
 \[\left\langle {\left. {{P_{X^*}}({x^{k + 1}}) - {P_{X^*}}({x^k}),{x^{k + 1}} - {P_{X^*}}({x^{k + 1}})} \right\rangle }\ge 0. \right.\]
 Therefore
 \begin{align}
      \frac{1}{2}h_{k+1}^2-\frac{1}{2}h_k^2&\nonumber\le
      \left\langle {\left.{x^{k + 1}}-{{x^k},{x^{k + 1}} - {P_{X^*}}({x^{k + 1}})} \right\rangle } \right.\\
      &\nonumber\le \left\lvert {\left\langle {{x^{k + 1}}-{x^k},{x^{k + 1}} - {P_{{X^*}}}({x^{k + 1}})} \right\rangle } \right\rvert\\
&\nonumber\le\left\|{x^{k + 1}}-{x^k} \right\|\, \left\|{x^{k + 1}} - {P_{X^*}}({x^{k + 1}})\right\|\\
&\label{eq:distance}= \left\|{x^{k + 1}}-{x^k}\right\|h_{k+1}.
 \end{align}
 Since $h_{k+1}\ne0$ the last inequality follows that 
 \[0 \le \frac{{{{(h_{k + 1}^2 - h_k^2)}^ + }}}{{h_{k+1}}} \le 2\left\|{x^k} - {x^{k + 1}}\right\|.\]
  Finally, the proof is completed by part (a) of Lemma \ref{lem:l2}.
 \end{proof}

\begin{rem}
We would also like to point out that if in Theorems \ref{thm:deltainfty} and Proposition \ref{prop:deltainfinty} we had  $\Omega:=VI(T,\Fix(W))=X^*$ then the trilevel optimization problem \eqref{prob:tri-levelintro} would have a solution.  
\end{rem}

\begin{rem}
Our assumptions on $\alpha_k$ and $\beta_k$ are weaker that of assumptions \cite{mainge2007strong}, \cite{moudafi2007krasnoselski}. For instance, consider that  $\mathop {\lim }\limits_k \frac{{{\beta _k}}}{{{\alpha _k}}}$ need not exist. Instead, we consider $\mathop {\lim }\limits_k \sup \frac{{{\beta _k}}}{{{\alpha _k}}}$. See Example~\ref{withoutlimit} below, where there is no limit $\frac{{{\beta _k}}}{{{\alpha _k}}}$, but our results still apply.
\end{rem}
\begin{example}\label{exam:delta}
\label{withoutlimit}
Let ${\beta _k} = \frac{1}{{(k+1)(3 + {{( - 1)}^k})}},\,\,\,\,{\alpha _k} = \frac{1}{{(k+1)(2 + {{( - 1)}^k})}}$
clearly $\mathop {\lim }\limits_k \frac{{{\beta _k}}}{{{\alpha _k}}}$ does not exist, but $\mathop {\lim \sup }\limits_k \frac{{{\beta _k}}}{{{\alpha _k}}} = \frac{3}{4}$ and also 
\[\,\sum\limits_{k = 1}^\infty  {{\alpha _k}}  = \infty ,\,\,\,{\alpha _k} \to 0.\]
\end{example}
Now, we are ready to give an example related to the step-sizes $\alpha_k$ and $\beta_k$  that guarantee the convergence $\{x^k\}$  of all of our results. Note that in all cases $\mathop {\lim }\limits_k \frac{{{\beta _k}}}{{{\alpha _k}}}$ may not exist. 
\begin{example} \label{ex:table}
Consider ${\alpha _k} = \frac{1}{{{(k+1)^\lambda (2+(-1)^k)}}}$ and ${\beta _k} = \frac{1}{{{(k+1)^\gamma(3+(-1)^k) }}}$    with $\lambda>0,~\gamma>0$. Now for large sufficient $k$, we have the following estimation \[\lvert{\alpha _{k + 1}} - {\alpha _k}\rvert\approx \frac{1}{{{(k+1)^\lambda }}},\,\,\,\,\,\,\lvert{\beta _{k + 1}} - {\beta _k}\rvert \approx \frac{1}{{{(k+1)^\gamma }}}.\,\] 
It is easy to check that
\[\delta=\mathop {\lim \sup }\limits_k \frac{{{\beta _k}}}{{{\alpha _k}}} = \left\{ \begin{array}{l}
 0\,\,\,\,\,\,\,\,\,\lambda  < \gamma , \\ 
 1\,\,\,\,\,\,\,\,\,\lambda  = \,\gamma ,\, \\ 
  + \infty \,\,\,\lambda  > \,\gamma \,.\, \\ 
 \end{array} \right.\]
Furthermore,\\
 \begin{itemize}
 \item  Assumption \ref{ass:ii} holds when $0<\lambda\le 1$.\\
  \item Assumption \ref{ass:iv} holds when $0<\lambda+\gamma\le 1$.\\
 \item  Assumption \ref{ass:v} holds when $0<\lambda, \gamma<1$.\\
  \item  Assumption $(A_4)$ holds when $\lambda\le 2\gamma$.
  \end{itemize}
  Remark that for the case $\widetilde{\delta}=\infty$, it is sufficient to consider ${\alpha _k} = \frac{1}{{{{(k + 1)}^\lambda }}}$ and ${\beta _k} = \frac{1}{{{{(k + 1)}^\gamma }}}$ with $\mathop {\lim }\limits_k \frac{{{\beta _k}}}{{{\alpha _k}}} = \infty$ for $\lambda>\gamma$.
  Now we are ready to present a taxonomy of assumptions with respect to $\alpha_k$ and $\beta_k$, as referenced in Table 2
  \begin{itemize}
      \item (a): $\lambda<\gamma$.
      \item (b): $0<\lambda\le\gamma<1$ or $0<\gamma<\lambda<1$.
      \item (c): $0<\lambda\le\gamma<1$ and $0<\lambda+\gamma<1$.
      \item (d): $0<\gamma<\lambda\le1$, $0<\lambda+\gamma\le1$ and $\lambda\le2\gamma$.
      \item (e): $0<\gamma<\lambda\le1$, $0<\lambda+\gamma\le1$. 
      \item (f): $0<\gamma=\lambda<1$.
  \end{itemize}
\end{example}

\section{Convergence Rate Analysis for Trilevel Optimization Problems}

In this section, we present the main result of this paper. This addresses the rate of convergence of the sequence $\{x^k\}$ generated by Algorithm \ref{algorithm:trilevel} with a particular choice of step-sizes. 

\subsection{Technical lemmas}

The technical lemma which we state next, and for which we refer to \cite{Sabach2017},  will play a crucial role in the convergence analysis.

\begin{lem} \cite[Lemma 3]{Sabach2017}\label{lem:convergesrate}
 Let $M>0$. Suppose that $\{a_k\}$ is a sequence of non-negative real numbers which satisfy $a_1\le M$ and \[{a_{k + 1}} \le (1 - \gamma {b_{k + 1}}){a_k} + ({b_k} - {b_{k+1}}){c_k}\,\,\,\,\,\,\,\ k \ge 1, \] where $\gamma\in (0,1], \{b_k\}$ is a sequence defined as $b_k:=\min\{\frac{2}{\gamma k},1\}$ and $\{c_k\}$ is a sequence of real numbers such that $c_k\le M<\infty$. 
 
 Then, the sequence $\{a_k\}$ satisfies \[a_k\le\frac{MJ}{\gamma k}\,\,\,\,\,\,\ k\ge 1,
 \,\,\,\text{ where } J = \left\lfloor {\frac{2}{\gamma }} \right\rfloor.
 \]
 \end{lem}
 The next result will be useful for the rate of convergence.
 \begin{lem}\label{lem:boundlipsshitz}
One has the following
\begin{align}
    \phi_i(x - t \psi_i(x)) - {\phi_i }(y) \le \frac{1}{2t}\|x - y\|^2,\,\,\,\,\,\,\,\,\,\,\forall x,y \in\RR^n,\,\,t\in(0,\frac{1}{L_{f_{i}}}]
\end{align}
where
$\phi_i(x):=f_i(x)+g_i(x)$ and $\psi_i (x):=\frac{1}{t}(x-\prox_{tg_{i}}(x-t\nabla f_i(x)))$ and $i=1,2$ 
\end{lem}  
\begin{proof}
 Let $i=1$. Using Lipschitz continuity of $f_1$ with parameter $L_{f_1}$ it is well-known that convexity of $f_1$ is equivalent to 
\begin{equation}
    {f_1}(y) \le {f_1}(x) + \nabla {f_1}{(x)^T}(y - x) + \frac{L_{f_1}}{2}\|x - y\|^2,  ~~~ \forall x,y\in\RR^n.
\end{equation}
Assume that $x\in\RR^n$ and $t\in(0,\frac{1}{L_{f_1}}]$ be given.~Plugging $y=x-t\psi_1 (x)$ in the previous inequality one obtains that
\begin{align}
    {f_1}(x - t\psi_1 (x)) &\nonumber\le {f_1}(x) + \nabla {f_1}{(x)^T}((x - {t}\psi_1 (x)) - x) + \frac{{{L_{f_1}}{t}^2}}{2}\|\psi_1 (x)\|^2\\
    &\nonumber\le{f_1}(x) - {t}\nabla {f}_1{(x)^T}(\psi_1 (x)) + \frac{{{t}}}{2}\|\psi_1 (x)\|^2.
\end{align}
Now from $x-t\psi_1 (x)=\prox_{tg_1}(x-t\nabla f_1(x)),$ we get to \[\psi_1 (x)-\nabla f_1(x)\in\partial g_1(x-t\psi_1 (x)).\]
Therefore, 
\begin{align}\label{eq:subdiff}
    {g_1}(y) - {g_1}(x - {t}\psi_1 (x)) \ge {(\psi_1 (x) - \nabla {f_1}(x))^T}(y - x + {t}\psi_1 (x)).
\end{align}
Now by simplifying and taking into account \eqref{eq:subdiff} one has  
\begin{align}
&\nonumber{\phi_1}(x - t\psi_1 (x)) = {f_1}(x - {t}\psi_1 (x)) + {g_1}(x - t\psi_1 (x))\\
&\nonumber\nonumber\nonumber\le{f_1}(x) - t\nabla {f_1}{(x)^T}(\psi_1 (x)) + \frac{{{t}}}{2}\|\psi_1 (x)\|^2 + {g_1}(x - t\psi_1 (x))\\
&\nonumber\nonumber\le{f_1}(y) - \nabla {f_1}(x){^T}(y - x) - t\nabla {f_1}{(x)^T}(\psi_1 (x)) + \frac{{{t}}}{2}\|\psi_1 (x)\|^2 + {g_1}(x - t\psi_1 (x))\\
&\nonumber\le{f_1}(y) -\nabla{f_1}(x){^T}(y - x) - t\nabla {f_1}{(x)^T}(\psi_1 (x)) + \frac{{{t}}}{2}\|\psi_1 (x)\|^2\\
&\nonumber+g_1(y)-{(\psi_1 (x) - \nabla {f_1}(x)^T}(y - x + t\psi_1(x))\\
&\label{eq:boundlipschitz}\le\phi_1(y)-\|\psi_1(x)\|\|x-y\|-t\|\psi_1(x)\|^2\\
&\nonumber\le \phi_1(y) + \frac{1}{{2{t}}}[\|x - y\|^2 - \|(x - y) -t\psi_1(x)\|^2]\le{\phi_1}(y) + \frac{1}{{2{t}}}\|x - y\|^2,
\end{align}
and therefore it follows \[ \phi_1(x - t \psi_1(x)) - {\phi_1 }(y) \le \frac{1}{2t}\|x - y\|^2,\,\,\,\,\,\,\,\,\,\,\forall x,y \in\RR^n,\,\,t\in(0,\frac{1}{L_{f_1}}].\] 
\end{proof}
We are now in a position to derive the following result which appeared in a similar form \cite{beck2009fast,Sabach2017}, however for our context the proof had to be modified.
\begin{prop}\cite[Proposition1]{Sabach2017}\label{prop:rateconvergence}
Let $x\in\RR^n$ and denote $x^+=T_t(x).$ Then 
\[\phi_2(x^+) - \phi_2(u) \le \frac{1}{t}\left\langle {x - x^+,x - u} \right\rangle  - \frac{1}{{2t}}\|x - x^+\|^2,\,\,\,\,\,\,\,\forall (u,t) \in\RR^n \times (0,\frac{1}{L_{f_2}}].\]

and also if $z=W_s(x)$. Then 
\[\phi_1(z) - \phi_1(u) \le \frac{2}{s}\left\langle {x - z,x - u} \right\rangle  - \frac{1}{{2s}}\|x - z\|^2,\,\,\,\,\,\,\,\forall (u,s) \in\RR^n \times (0,\frac{1}{L_{f_1}}].\]

\end{prop}
 \begin{proof}
We will just prove the first part. The second part can be proved by the same method. Assume that $x^+=T_t(x)$ so we have 
\[{\psi _2}(x) = \frac{1}{t}(x - \prox{_{t{g_2}}}(x - t\nabla {f_2}(x)) = \frac{1}{t}(x - {T_t}(x)) = \frac{1}{t}(x - {x^ + }).\]
From \eqref{eq:boundlipschitz} of Lemma \ref{lem:boundlipsshitz} we obtain 
\[{\phi _2}(x - t{\psi _1}(x)) - {\phi _2}(u) \le  - \left\| {{\psi _1}(x)} \right\|\left\| {x - u} \right\| - t{\left\| {{\psi _1}(x)} \right\|^2} \le \left\langle {{\psi _1}(x),x - u} \right\rangle  - \frac{t}{2}{\left\| {{\psi _1}(x)} \right\|^2},\]
and then by plugging $\psi_1(x)=\frac{x-x^+}{t}$ in the previous inequality one get 
\[\phi_2(x^+) - \phi_2(u) \le \frac{1}{t}\left\langle {x - x^+,x - u} \right\rangle  - \frac{1}{{2t}}\|x - x^+\|^2,\,\,\,\,\,\,\,\forall (u,t) \in\RR^n \times (0,\frac{1}{L_{f_2}}].\] 
and the desired result follows. 
\end{proof}

Now, we set up the sequences $\alpha_k$ and $\beta_k$ and define the constant $J$ as \textcolor{black}{\begin{equation}\label{eq:alpha=beta}
{\alpha _k} =\min \left\{ {\frac{2}{{(1 - r)k }},1} \right\},\,\,\,\beta_k=\frac{\alpha_{k}-\alpha_{k+1}}{2(2-\alpha_k)},\,\ J = \left\lfloor {{{\frac{2}{{(1 - r)}}}}} \right\rfloor ,{\mkern 1mu} {\mkern 1mu} {\mkern 1mu} {\mkern 1mu} {\mkern 1mu} {\mkern 1mu} {\mkern 1mu} {\mkern 1mu} k \ge 1
 \end{equation}}
 where $r\in (0,1]$. Clearly, Assumptions  \ref{ass:ii}, \ref{ass:iv}, \ref{ass:v} are satisfied under \eqref{eq:alpha=beta}. 
 
To begin, we present the following lemma, which plays a key role in the sequel. 
 \begin{lem} \label{lem:36}
 Assume that $\{x^k\}$, $\{y^k\}$,~$\{z^k\}$ and $\{v^k\}$ be the sequences generated by  Algorithm \ref{algorithm:trilevel} and also $x\in\Fix(W)$ be given, defining $y=T(x)$ and $v=S(x)$. Then, for every $k\ge 1$ the following relations hold true. 
 \begin{equation}\label{rateparts}
     \left\{ \begin{array}{l}
 \|{y^{k + 1}} - {y^k}\| \le \|{x^k} - {x^{k - 1}}\|, \\ 
 \|{z^{k + 1}} - {z^k}\| \le \|{x^k} - {x^{k - 1}}\|, \\ 
 \|{v ^{k + 1}} - {v ^k}\| \le r\|{x^k} - {x^{k - 1}}\|, \\ 
 \|{z^{k + 1}} -  x \| \le \|{x^k} -  x \|, \\
 \|{y^{k + 1}} -  y \| \le \|{x^k} -  x \|, \\
 \|{v ^{k + 1}} - v \| \le r\|{x^k} - x \|, \\ 
 \end{array} \right.
 \end{equation}
 and there exists positive constants $C_S,C_T$ and $C_x$ so that
 \begin{equation}
    \left\{ \begin{array}{l}
 \|{y^k} - {z^k}\| \le {C_T} + 2{C_{x }}, \\ 
 \|{y^k} - {v ^k}\| \le {C_S} +C_T+2{C_{ x }}, \\ 
 \|{v ^k} - {z^k}\| \le  {C_S} + 2{C_{ x }}. \\ 
 \end{array} \right.
 \end{equation}
\end{lem}   
\begin{proof}
All parts are a direct consequence of non-expansively of $T,W$ and the contraction property of $S$ and Lemma \ref{lem:boundedness}. 
\end{proof}
\begin{lem}\label{lem:boundedrate}
Let $\{x^k\}$, $\{y^k\}$,$\{z^k\}$ and $\{v^k\}$  be sequences generated by the Algorithm \eqref{algorithm:trilevel}, where $\{\alpha_k\}$ and $\{\beta_k\}$ are defined by \eqref{eq:alpha=beta}. Then for every $x\in \Fix(W)$ one has 
\begin{equation}\label{eq:inequalityrate}
  \left\{ \begin{array}{l}
 \|{x^k} - {x^{k - 1}}\| \le \frac{{({C_S} + 2{C_T} + 5{C_ x})J}}{{(1 - r)k }} \\ 
 and \\ 
 \|{z^k} - {x^{k - 1}}\| \le \frac{{({C_S} + 2{C_T} + 5{C_x})(J + 2)}}{{(1 - r)k }},\,\,\,\, \\ 
 \end{array} \right.
\end{equation}
where $C_S,C_T,C_x$ are defined in Lemma \ref{lem:boundedness} and $J = \left\lfloor {{{\frac{2}{{(1 - r)}}}}} \right\rfloor$.
\end{lem}
\textcolor{black}{\begin{proof}
One can write
\[\begin{array}{l}
 {x^{k + 1}} - {x^k} = \alpha _{k+1}[v^{k+1} - v^k] \\ 
\,\,\,\,\,\,\,\,\,\,\,\,\,\,\,\,\,\,\,\,\,\,\,\,\,\,\,\,\,\,\, + (1-\alpha_{k+1})[\beta _{k+1}(y^{k+1}-y^k) + (1 - \beta _{k+1})(z^{k+1}-z^k)] \\ 
\,\,\,\,\,\,\,\,\,\,\,\,\,\,\,\,\,\,\,\,\,\,\,\,\,\,\,\,\,\,\,\, + (\alpha _{k+1} - \alpha _k)[v^k - \beta _{k + 1}y^k - (1 - \beta _k)z^k] \\ 
\,\,\,\,\,\,\,\,\,\,\,\,\,\,\,\,\,\,\,\,\,\,\,\,\,\,\,\,\,\,\,\, + (1 - \alpha _{k+1})(\beta _{k+1} - \beta _k)[y^k- z^k]. \\ 
 \end{array}\]
 Now, one gets that 
 \begin{align*}
\|{x^{k + 1}} - {x^k}\| &\le (1 - (1 - r)\alpha _{k+1})\|{x^k} - {x^{k - 1}}\|\\
&+ (\alpha _k - \alpha _{k+1})\|v^k- \beta _{k +1}y^k - (1 - \beta _k)z^k\\
&+(1 - \alpha _{k+1})(\beta_{k+1}-\beta_k)(y^k - z^k)\|\\&
\le(1-(1-r)\alpha_{k+1})\|x^k-x^{k-1}\|+(\alpha_k-\alpha_{k+1})c_k
 \end{align*}
 as well as one can easily follow that 
\begin{align}
 c_k: &\label{eq:rate}= \|v^k- \beta _{k +1}y^k - (1 - \beta _k)z^k+(1 - \alpha _{k+1})(\beta_{k+1}-\beta_k)(y^k - z^k)\|\\ 
 &\nonumber=\|(v^k - x ) - {\beta _{k }}({z^{k}} - x ) - (1 - {\beta _{k  }})({y^{k}} -x )\\ 
 &\nonumber+ (1 - {\alpha _{k+1}})(({z^{k }} -  x ) - ({y^{k}} - x ))\| \le \|{v ^{k}} -  x \| + {\beta _{k}}\|{z^{k }} -  x \|\\ 
 &\nonumber+ (1 - {\beta _{k}})\|{y^{k}} - x \| + (1 - {\alpha _{k+1}})\|{z^{k }} -  x \| + (1 - {\alpha _k})\|{y^{k }} -  x \|\\ 
 &\nonumber\le ({C_S} + {C_{ x }}) + ({C_T} + {C_{ x }}) + {C_{x }} + ({C_T} + {C_{ x }}) + {C_{ x }} = {C_S}+{ + 2{C_T} + 5{C_x}}.
\end{align}
 Moreover, 
 \[\|{x^1} - {x^0}\| = \|({x^1} -  x ) - ({x^0} - x )\| \le \|{x^1} - x \| + \|{x^0} -  x \| \le 2{C_{ x }} \le {{C_S} + 2{C_T} + 5{C_{ x }}},\]
 therefore all  hypotheses  of Lemma \ref{lem:convergesrate} are hold. Hence, the rate of convergence $\{\|x^{k+1}-x^k\|\}$ is immediately implied by setting $a_k=\|x^{k}-x^{k-1}\|$,$ b_k=\alpha_k$, $\gamma=1-r$ and $c_k$ as \eqref{eq:rate}.  
 By the following arguments, the rate for $\{\|z^k-x^{k+1}\|\}$ can be derived
 \begin{equation}
     \begin{array}{l}
 \|{z^k} - {x^{k}}\| = \|{z^k} - \left( {{\alpha _k}{v ^k} + (1 - {\alpha _k}){\beta _k}{y^k} + (1 - {\alpha _k})(1 - {\beta _k}){z^k}} \right)\| \\ 
\,\,\,\,\,\,\,\,\,\,\,\,\,\,\,\,\,\,\,\,\,\,\,\,\,\,\,\, = \,\|{\alpha _k}({z ^k} - {v^k}) + (1 - {\alpha _k}){\beta _k}({z^k} - {y^k})+(1-\alpha_k)(1-\beta_k)(z_k-z_k)\| \\ 
\,\,\,\,\,\,\,\,\,\,\,\,\,\,\,\,\,\,\,\,\,\,\,\,\,\,\,\,\, \le {\alpha _k}\|{v ^k} - {z^k}\| + (1-\alpha_k)\beta_k\|{z^k} - {y^k}\|= \alpha_k\|v^k-z^k\|+\alpha_k\|z^k-y^k\|\\
\,\,\,\,\,\,\,\,\,\,\,\,\,\,\,\,\,\,\,\,\,\,\,\,\,\,\,\,\,\,=\alpha_k\left(\|v ^k - z^k\|+\|z^k-y^k\|\right)\le
\frac{2}{{(1 - r)k }}\left({C_S} + {C_T} + 4{C_x}\right), \\ 
 \end{array}
 \end{equation}
 where have used the fact that $(1-\alpha_k)\beta_k\le\alpha_k$,
 and we also have 
 \begin{align*}
 \|{z^k} - {x^{k - 1}}\|&\le \|{z^k} - {x^k}\| + \|{x^k} - {x^{k - 1}}\|\\ 
 &\le \frac{2}{{(1 - r)k }}({C_S} + {C_T} + 4{C_x}) + \frac{{({C_S} + 2{C_T} + 5{C_x})J}}{{(1 - r)k }}\\ 
 &\le \frac{{({C_S} + 2{C_T} + 5{C_x})(J + 2)}}{{(1 - r)k }}. 
 \end{align*}
\end{proof}}
\subsection{The Main Result}

Now, we are in a position to conclude our main result concerning the rate of convergence for convex trilevel optimization. Having proved that $\|z^k-x^{k-1}\|\to0$ as $k\to\infty$ and considering the lower semi-continuity of $\phi_1$, one obtains that $\{\phi_1(z^k)\}_{k\in\NN}$ converges to the optimal value.
Furthermore, this implies the convergence of the sequence $\{\phi_1(x^k)\}_{k\in\NN}$ to the same value. 

We note that the same argument holds for the sequence $\{\phi_2(z^k)\}_{k\in\NN}$.  The following theorem presents the convergence rate in function values to their optima:
\textcolor{black}{\begin{thm}\label{thm:rateconvergence}
Let $\{x^k\},\{v^k\},\{z^k\}$ and $\{y^k\}$ be sequences generated by Algorithm \eqref{algorithm:trilevel}, where $\alpha_k$ is proposed by \eqref{eq:alpha=beta}. Then
\begin{align}
\label{thm42a}
{\phi _1}({z^{k+1}}) - {\phi _1}({x^*}) \le \frac{(J+2)(C_S+2C_T+5C_{x^*})\sqrt{JC_{x^*}(C_T+C_{x^*})}}{s\sqrt{1-r}(1-r)(k+1)\sqrt{k}}~~ \forall (s,k) \in (0,\frac{1}{L_{f_2}}] \times \NN,
 \end{align}
 where $C_S,C_T,C_{x^*}$ are the same constants as in Lemma \ref{lem:boundedness}, $J$ is defined in \eqref{eq:alpha=beta},
 Furthermore, one has 
 \begin{align}
 &{\phi _2}({y^{k+1}}) - {\phi _2}({x^*}) \le \frac{{{C_{{x^*}}}J}}{{2t(1 - r)k}},~~~~~~~~~~~~~~\forall (t,k)\in (0,\frac{1}{L_{f_1}} ]\times\NN \label{eq:rate54}\\&
 \phi_i(x^k)-\phi_i(x^*)\le\sqrt{\frac{4(C_T+C_{x^*})J}{C_{x^*}(1-r)k}}~~~~~~~~~~~~{\rm for}~ i=1,2~ {\rm and}~ k\in\NN
 \label{eq:rate55}\\&
\omega ({x^k}) - \omega ({x^*}) \le \left({L_\omega } + \frac{{L_\omega ^2}}{{2\mu }}r{C_{{x^*}}}\right)r\sqrt{\frac{C_{x^*}(C_T+C_{x^*})J}{\gamma k}}. 
\end{align}
\end{thm}
\begin{proof}
Since $x^*$ is a solution of tri-level Problem  \ref{prob:tri-levelintro} of  Theorem \ref{thm:deltazero}, the following result was obtained:
\begin{equation}\label{eq:tri.r}
    \|x^{k+1}-x^*\|^2\le (1-(1-r)\alpha_k)\|x^k-x^*\|^2+2(1-\alpha_k)\beta_k\langle T(x^*)-x^*,x^{k+1}-x^*\rangle.
\end{equation}
Let us take $$a_k=\|x^k-x^*\|^2,~c_k=\|\langle T(x^*)-x^*,x^{k+1}-x^*\rangle\|$$
and $\alpha_k$,$\beta_k$ be as in \eqref{eq:alpha=beta}. From \eqref{eq:tri.r}, it 
follows that 
\begin{align*}
    a_{k + 1} &\le \left(1 - (1 - r)\alpha _k\right)a_k+2(1-\alpha_k)\beta_kc_k\\&
    \le\left(1-(1-r)\alpha_{k+1}\right)a_k+(\alpha_k-\alpha_{k+1})c_k
\end{align*}
where we used the facts $2(1-\alpha_k)\beta_k\le\alpha_k-\alpha_{k+1}$, and $\alpha_{k+1}\le\alpha_k$. By Lemma \ref{lem:boundedness}, we then have $c_k\le(C_T+C_{x^*})C_{x^*}$. By utilizing Lemma \ref{lem:convergesrate},   we obtain 
\begin{equation}\label{eq:ratex}
   {\left\| {{x^k} - {x^*}} \right\|^2} \le \frac{{{C_{x^*}(C_T+C_{{x^*}}})J}}{{(1 - r)k}},~~~~~\forall k\in\NN.
\end{equation}
Let us consider the first assertion \eqref{thm42a}. According to Proposition \ref{prop:rateconvergence} and $z^{k+1}=W(x^k)$, for every step-size $s\le\frac{1}{L_{f_1}}$ the following inequality holds
\begin{align}\label{rate1}
    \phi_1({z^{k + 1}}) - \phi_1({x^*}) & \le \frac{1}{s}\left\langle {{x^k} - {z^{k + 1}},{x^k} - {x^*}} \right\rangle  - \frac{1}{{2s}}\|{x^k} - {z^{k + 1}}\|^2 \\ & \le \frac{1}{s}\left\langle {{x^k} - {z^{k + 1}},{x^k} - {x^*}} \right\rangle. 
\end{align}
Combining with \eqref{eq:inequalityrate} and \eqref{eq:ratex} for $x^*\in X^*\subseteq \Fix(W)=Y^*$, one obtains 
\begin{align}
     \left\langle {{x^k} - {z^{k + 1}},{x^k} - {x^*}} \right\rangle  &\nonumber\le \|{x^k} - {z^{k + 1}}\|.\|{x^k} - {x^*}\|\\
    &\label{rate2}\le\frac{(J+2)(C_S+2C_T+5C_{x^*})\sqrt{JC_{x^*}(C_T+C_{x^*})}}{\sqrt{1-r}(1-r)(k+1)\sqrt{k}}.
\end{align}
Thus, the assertion \eqref{thm42a} follows from \eqref{rate1} and \eqref{rate2}.\\
Now, we obtain the rate of convergence for $\phi_1(y^k)-\phi_1(x^*)$. In Algorithm \ref{algorithm:trilevel}, we have $y^{k+1}=T(x^k)$ and so using Lemma \ref{lem:boundlipsshitz}, one gets that 
\begin{equation}\label{eq:ratephione}
    {\phi _1}({y^k}) - {\phi _1}({x^*}) = {\phi _1}(T({x^{k}})) - {\phi _1}({x^*}) \le \frac{1}{2t}\| {{x^{k}} - {x^*}}\|^2.
\end{equation}
Plugging inequality \eqref{eq:ratex} into  \eqref{eq:ratephione} gives the desired assertion \eqref{eq:rate54}.\\
To establish the assertion \eqref{eq:rate55}, since $\phi_i$, $(i=1,2)$ is convex and bounded above on compact set $\mathcal{B}(x^*,C_{x^*})$, by invoking  \cite[Theorem 2.1.10]{borwein2010convex} one concludes that $\phi_i$ is Lipschitz on this set. Note thanks to $x^k\in\mathcal{B}(x^*,C_{x^*})$, one has 
\begin{align}
    \phi_i(x^k)-\phi_1(x^*)\le\frac{2}{C_{x^*}}\|x^k-x^*\|
\end{align}
which combined with \eqref{eq:ratex} gives the desired assertion.\\ 
Finally, to bound the rate of convergence $\omega(x^k)-\omega(x^*)$, we take into consideration \eqref{rateparts},~\eqref{eq:ratex}, and strong convexity of $\omega$, for all $k\in\NN\setminus\{1\}$. We obtain: 
\begin{align*}
\omega ({x^k}) - \omega ({x^*}) &\le \nabla \omega {({x^*})^T}({x^k} - {x^*}) + \frac{1}{{2\mu }}{\left\| {\nabla \omega ({x^k}) - \nabla \omega ({x^*})} \right\|^2}\\
&\le\left({L_\omega } + \frac{{L_\omega ^2}}{{2\mu }}\left\| {{x^k} - {x^*}} \right\|\right)\left\| {{x^k} - {x^*}} \right\|\\
& \le \left({L_\omega } + \frac{{L_\omega ^2}}{{2\mu }}r{C_{{x^*}}}\right)r\sqrt{\frac{C_{x^*}(C_T+C_{x^*})J}{\gamma k}}.
\end{align*}
\end{proof}}
\begin{rem}
It is worth pointing out that step-size $\{\alpha_k\}$ depends on the parameter $r$, which needs to be chosen so that the map $S$ is a contraction. Notice, however, that knowing $L_{\omega}$ and $\mu$, one can consider $r$ such that $r\in ({0,\frac{2}{{{L_\omega } + \mu }}}]$. In this case, the map $S$ is guaranteed to be a contraction.
\end{rem}
\section{An Extension to Multilevel Optimization Problems}
In this section, we extend our results to a multilevel convex optimization problem wherein we have an arbitrary number of nested minimization problems:  

\textcolor{black}{\[\left\{ \begin{array}{l}
 \mathop {\arg\min }\limits_{x \in X_N^*} \omega (x), \\
X_N^* = \mathop {\arg\min }\limits_{x \in X_{N-1 }^*} [{f_{N }}(x) + {g_{N }}(x)],
  \\ 
 . \\ 
 . \\ 
 . \\ 
 X_2^* = \mathop {\arg\min }\limits_{x \in X_1^*} [{f_2}(x) + {g_2}(x)]
 ,
 \\ 
 X_1^* = \mathop {\arg\min }\limits_{x \in\mathbb{R}^n} [{f_1}(x) + {g_1}(x)]
,
 \end{array} \right.\]}
It is then natural to define the following algorithm
\begin{equation}\label{al.multi}
    {x^{k + 1}} = {\alpha _k}^{(0)}S({x^k}) + {\alpha _k}^{(1)}{T_1}({x^k}) + ... + {\alpha _k}^{(N)}{T_N}({x^k}),
\end{equation}
\textcolor{black}{in which $T_1,T_2,...,T_N$ 
are computed as the corresponding equation},  i.e.,
\eqref{prob:multi-level}, i.e.,
\[{T_{i}}(x) =\prox {_{{t_i}{g_i}}}[x-t_i\nabla {f_i}(x)],\,\,\forall i \in \{ 1,2,...,N \}, \]
and where the step-sizes $\left( {{\alpha _k}^{(0)},{\alpha _k}^{(1)},{\alpha _k}^{(2)},...,{\alpha _k}^{(N)}} \right)$ satisfy the following:\\
$(P_1)$ for all $k\in\NN$, one has $\sum\limits_{j =0}^N {\alpha _k^{(j)}}  = 1$ .\\
\textcolor{black}{$(P_2)$ $\mathop {\lim }\limits_{k \to \infty } \alpha _k^{(0)} = 0\,$, $\sum\limits_{k = 1}^{ + \infty } {\alpha _k^{(0)}}  = \infty $  and, $\mathop {\lim \sup }\limits_k \frac{1}{{\alpha _k^{(0)}}}\sum\limits_{i = 2}^{ N} {\alpha _k^{(i)}}  = {\delta ^*} \in [0, + \infty )$}.\\
$(P_3)$ $X_N^*\ne\emptyset$, $\mathop {\lim }\limits_{k \to \infty } \alpha _k^{(1)} = 1$ and for all $j\in\{0,2,...,N\}$ one has $\mathop {\lim }\limits_{k \to \infty } \alpha _k^{(j)} = 0$.

The following facts hold:
\begin{itemize}
    \item $X_N^* \subseteq X_{N-1}^* \subseteq ... \subseteq X_1^*=\Fix(T_1).$
    \item  for every $t_i\in(0,\frac{1}{L_{f_i}}]$  and $i\in\{1,2,...,N\}$ one has   $\widetilde{X}_i^*=\Fix(T_i)$ where $$\widetilde{X}_{i}^{*}= \mathop {\arg \min }\limits_{x \in\RR^n} [{f_i}(x) + {g_i}(x)].$$
    \item  for every $x\in VI(T_{i+1},\Fix(T_i))$ one has
    \[{\phi _i}({T_{i + 1}}(x)) \le {\phi _i}(y)\,\,\,\,\,\,\forall y \in \Fix({T_i}),\]
    where ${\phi _i}(x): = {f_i}(x) + {g_i}(x)$ for all $i\in\{1,2,...,N\}$.
    \item for each $i\in\{1,2,...,N-1\}$ one has $X_{i+1}^*\cap\Fix(T_{i+1})=\Fix(T_{i+1})\cap\Fix(T_i)$
    \item for $i\in\{1,2,..,N\}$, one has \begin{equation}\label{eq:generalcase}
        {\phi _i}({T_i}(x)) - {\phi _i}(y) \le \frac{1}{{2t}}{\left\| {x - y} \right\|^2},\,\,\,\,\,\,\forall x,y \in \RR^n,\,\,\forall t \in (0,\frac{1}{{{L_{{f_i}}}}}],
    \end{equation}
     this follows from Lemma \ref{lem:boundlipsshitz}.
      \item for every $x \in \text{dom}(\partial {\phi _i})$ one has
    \[\|x - {T_i}(x)\| \le d(0,\partial {\phi _i}),\]
    where $\dom(\partial\phi_i)=\{x\in\RR^n:~~\partial\phi_i(x)\ne\emptyset\}.$
\end{itemize}
\textcolor{black}{
\begin{assume}\label{ass:multilevelassumption}
(Quadratic growth condition)
Suppose now that $\phi_1$ grows quadratically (globally) away
from a part of its minimizing set $X_1^*=\Fix(T_1)$, i.e.,  $X_N^*$, meaning there is a real number  $\eta>0$ such that
\begin{align}
    \phi_1(x)\ge\phi_1^*+\frac{\eta}{2}{\rm dist}^2(x,X_N^*)~~~\forall
    x\in\Omega^*\backslash\Fix(T_1)
\end{align}
where $\Omega^*=\mathcal{B}(x_0,C_{x_0})$ for given $x_0\in\Fix(T_1)$ and $\phi_1^*$ represents the optimal value of  $\phi_1$.
\end{assume}}
In view of Fact \ref{Fact.interior} the following general result, however, holds true in having the qualification condition of being a non-empty interior $X_N^*$. We omit the proof. 
\textcolor{black}{
\begin{lem}\label{lem.multi}
If $X_N^*$ has non-empty interior then $X_N^*=\cap_{i=1}^N\Fix(T_i)$.
\end{lem}}
\begin{lem}\label{bounded.multi}
Let Assumption $(P_2)$ hold. Then the sequence $\{x^k\}$ generated by algorithm \eqref{al.multi} is bounded.
\end{lem}
\begin{proof}
Define \[{R_k}(x) = \frac{{{\alpha _k}^{(1)}}}{{1 - {\alpha _k}^{(0)}}}{T_1}(x) + \frac{{{\alpha _k}^{(2)}}}{{1 - {\alpha _k}^{(0)}}}{T_2}(x) + ....\frac{{{\alpha _k}^{(N)}}}{{1 - {\alpha _k}^{(0)}}}{T_N}(x),\,\,\,\,\,\,\,\forall k \in\NN.\]
Then, one can rewrite the sequence $\{x^{k+1}\}$ as 
\[{x^{k + 1}} = {\alpha _k}^{(0)}S({x^k}) + (1 - {\alpha _k}^{(0)}){R_k}({x^k}),\]
and since $R_k$ is a convex combination of non-expansive operators, then it is non-expansive. Now, as in Lemma \ref{lem:boundedness}, for every $x\in \Fix(T_1)$ one has
\[\|{x^{k + 1}} - x\| \le \max \left\{ \|{x^{{k_0}}} - x\|,\frac{1}{{1 - r}}(\|S(x) - x\| + {\delta ^*}\|{T_j}(x) - x\|)\right\}, \]
in which 
\[\|T_j(x)-x\|=\mathop {\max }\limits_{2 \le i \le N} \|T_i(x)-x\|\]
\end{proof}


\begin{lem}[\cite{bauschke2011convex}]\label{fixedpoinmulti}
Let $T_1,T_2,...,T_N$ be non-expansive mappings from $\mathbb{R}^n$ to itself such that $\cap_{i=1}^NT_i$ is non-empty and let $\lambda_1,\lambda_2,...,\lambda_n$ be real numbers such that $\sum_{i=1}^N \lambda_i=1$. Then \[\Fix\left(\sum\limits_{i = 1}^N {\lambda _i}{T_i}\right) = \bigcap\limits_{i = 1}^N {\Fix({T_i})} . \]
\end{lem}
\begin{lem}\label{lem:multi}
Let Assumption $(P_2)$ hold. Then for every $x\in \Fix(T_1)$, one has
\begin{align*}
    \|{x^{k + 1}} - x\|^2 &\le (1 - (1 - r){\alpha _k}^{(0)})\|{x^k} - x\|^2 + \alpha_k^{(0)}\left\langle S(x)-x,x^{k+1}-x\right\rangle\\&
    +\sum_{i=2}^N{{\alpha _k}^{(i)}\left\langle {{T_i}(x) - x,{x^{k + 1}} - x} \right\rangle }.
\end{align*}
\end{lem}
\begin{proof}
Suppose that $x\in\Fix(T_1)$ is given. First, consider 
\[\begin{array}{l}
 {C_k}: = {\alpha _k}^{(0)}(S({x^k}) - S(x)) + {\alpha _k}^{(1)}({T_1}({x^k}) - {T_1}(x)) + ... \\~~~+ {\alpha _k}^{(N)}({T_N}({x^k}) - {T_N}(x)), \\\text{and}\\ 
 {D_k}: = {\alpha _k}^{(0)}(S({x}) - x) + {\alpha _k}^{(2)}({T_2}({x}) - x) + ... + {\alpha _k}^{(N)}({T_{N}}({x}) - x). \\ 
 \end{array}\]
 Therefore, we have 
 \begin{equation}\label{en:multi}
     {C_k} + {D_k} = {x^{k + 1}} - x,\,\,\,\,\,\,\,\,\|{C_k}\| \le (1 - (1 - r){\alpha _k}^{(0)})\|{x^k} - x\|.
 \end{equation}
To see the second part, for all $k\in\NN$ we have 
\[\begin{array}{l}
 \left\| {{C_k}} \right\| \le r\alpha _k^{(0)}\left\| {{x^k} - x} \right\| + \alpha _k^{(1)}\left\| {{x^k} - x} \right\| + ... + \alpha _k^{(N)}\left\| {{x^k} - x} \right\| \\ 
 \,\,\,\,\,\,\,\,\,\,\,\, = (r\alpha _k^{(0)} + \alpha _k^{(1)} + ... + \alpha _k^{(N)})\left\| {{x^k} - x} \right\| \\ 
 \,\,\,\,\,\,\,\,\,\,\,\,\,\, = (1 - (1 - r)\alpha _k^{(0)})\left\| {{x^k} - x} \right\|. \\ 
 \end{array}\]
 Now, by plugging \eqref{en:multi} in the following inequality 
 \[\|{C_k} + {D_k}\|^2 \le \|{C_k}\|^2 + \left\langle {{D_k},{C_k} + {D_k}} \right\rangle, \]
 the assertion follows immediately.
\end{proof}

Now, we are in a position to present our main result regarding multi-level scenarios. 
\textcolor{black}{
\begin{thm}
    Assume that $X^*$ has non-empty interior. Moreover the following holds\\
    $(A_1^*)$ $\limsup_k\frac{\sum
_{i=1}^N\|\alpha_k^{(i)}-\frac{1-\alpha_k^{(0)}}{N}\|}{\alpha_k^{(0)}}=0$.\\
    Then $\{x^k\}$ convergence to some unique $x_N^*$ such that
\[\left\langle {x_N^* - S(x_N^*),y - x_N^*} \right\rangle  \ge 0,\,\,\,\,\,\,\,\forall y \in X_N^*.\]  
\end{thm}
\begin{proof}
   First, we note that by looking at Lemma \ref{lem.multi} we observe that 
   \[X^* = \Fix(\frac{\sum_{i=1}^NT_i}{N}).\]
   Let us now, consider the auxiliary sequence $$y^{k+1}=\alpha_k^{(0)}S(y^k)+(1-\alpha_k^{(0)})\frac{\sum_{i=1}^NT_i(y^k)}{N}.$$
By employing \cite[Theorem 3.2]{xu2004viscosity} along with Lemma \ref{fixedpoinmulti}, one establishes the convergence of the sequence ${y^k}$ to a specific point denoted as $x_N^*$. Consequently, we deduce the following:
\begin{align*}
    \|x^{k+1}-y^{k+1}\|&\le r\alpha_k^{(0)}\|x^k-y^k\|+\|\mathop\sum\limits_{i=1}^N\alpha_k^{(i)}T_i(x^k)-(1-\alpha_k^{(0)})\frac{\sum_{i=1}^NT_i(y^k)}{N}\|\\&
    \le (1-(1-r)\alpha_k^{(0)})\|x^k-y^k\|+\|\sum\limits_{i=1}^N(\alpha_k^{(i)}-\frac{1-\alpha_k^{(0)}}{N})T_i(y^k)\|\\&
    \le(1-(1-r)\alpha_k^{(0)})\|x^k-y^k\|+(\sum\limits_{i=1}^N\|\alpha_k^{(i)}-\frac{1-\alpha_k^{(0)}}{N}
    \|)\|T_i(y^k)\|
\end{align*}
using Lemma \ref{lemma:L1} follows that $\|x^k-y^k\|$ goes to zero as $k\to\infty$.
\end{proof}}
\textcolor{black}{\begin{thm}\label{thm:generalcase}
Let Assumption $(P_2)$ hold with $\delta^*=0$, together with \ref{ass:multilevelassumption}. Then the sequence $\{x^k\}$ generated by algorithm \eqref{al.multi} converges to some unique $x_N^*\in X_N^*$ such that 
\[\left\langle {x_N^* - S(x_N^*),y - x_N^*} \right\rangle  \ge 0,\,\,\,\,\,\,\,\forall y \in X_N^*,\]
and \[\mathop {\min }\limits_{x \in X_N^*} \omega (x) = \omega (x_N^*).\]
\end{thm}
\begin{proof}
    Let $x_N^*\in X_N^*$ be the unique fixed point of the contraction $P_{X_N^*}S$, namely the unique solution of $VI(S,X_N^*)$, i.e.,  
\begin{align}\label{trilevel0}
    \left\langle {{x_N^*} - S({x_N^*}),x - {x_N^*}} \right\rangle\ge 0,\,\,\,\,\,\,\,\forall x \in {X_N^*}.
\end{align}
Invoking Assumption $(P_2)$, one gives that $\{x^{k}\}$ is bounded, and so $x^k\in\Omega^*$. Furthermore, utilizing assumption \ref{ass:multilevelassumption}, it is straightforward to show that $w(x^k)\subset X_N^*$. Moreover, one can extract a convergent sub-sequence $\{x^{k_i}\}$ of $\{x^{k+1}\}$ or any sub-sequence thereof to $x_N^{'}\in X_N^*$, which holds by Lemma~\ref{lem:l2}, part c, and \eqref{trilevel0} so that
\begin{align}\label{eq:A*2}
\mathop {\lim \sup }\limits_k \left\langle {S({x_N^*}) - {x_N^*},{x^{k+1}} - {x_N^*}} \right\rangle  &\nonumber= \mathop {\lim }\limits_i \left\langle {S({x_N^*}) - {x_N^*},{x^{{k_i}}} - {x_N^*}} \right\rangle\\
= \left\langle {S({x_N^*}) - {x_N^*},x_N^{'} - {x_N^*}} \right\rangle  \le 0.
\end{align}
and now, using Lemma \ref{lem:multi} follows that
\begin{align} \label{eq:A*5}
   \|{x^{k + 1}} - x^*_N\|^2 \nonumber&\le (1 - (1 - r){\alpha _k}^{(0)})\|{x^k} - x^*_N\|^2 +\alpha_k^{(0)}\left\langle S(x_N^*)-x_N^*,x^{k+1}-x_N^*\right\rangle\nonumber\\&
   +\sum_{i=2}^N{{\alpha _k}^{(i)}\left\langle {{T_i}(x^*_N) - x^*_N,{x^{k + 1}} - x^*_N} \right\rangle }.
\end{align}
Next, set
\begin{align}\label{eq:A*6}
\left\{ \begin{array}{l}
 {a_k}= \|{x^k} - {x^*}\|^2, \\ 
 {\gamma _k}= (1 - r){\alpha^{(0)} _k}, \\ 
 \delta_k=\alpha_k^{(0)}\left\langle S(x_N^*)-x_N^*,x^{k+1}-x_N^*\right\rangle+\mathop\sum\limits_{i=2}^N{{\alpha _k}^{(i)}\left\langle {{T_i}(x^*_N) - x^*_N,{x^{k + 1}} - x^*_N} \right\rangle }.\\ 
 \end{array}\right.
 \end{align}
One has that
\[a_{k+1}\le (1-\gamma_k)a_k+\delta_k.\]
Additionally, utilizing the boundedness of $\{x^k\}$ together with $\delta^*=0$ it can be inferred that $\mathop {\lim \sup }\limits_k \frac{{{\delta _k}}}{{{\gamma _k}}} \le 0.$ Indeed, taking into account \eqref{eq:A*2} and $\delta^*=0$ gives
 \begin{align*}
\limsup_k\frac{\delta_k}{\gamma_k}&=\limsup_k\left[\frac{\alpha_k^{(0)}\left\langle S(x_N^*)-x_N^*,x^{k+1}-x_N^*\right\rangle+\mathop\sum\limits_{i=2}^N\alpha_k^{(i)}\left\langle T_i(x_N^*)-x_N^*,x^{k+1}-x_N^*\right\rangle}{\alpha_k^{(0)}}\right]\\&
\le\limsup_k\left\langle S(x_N^*)-x_N^*,x^{k+1}-x_N^* \right\rangle+C\limsup_k\frac{\sum_{i=2}^N\alpha_k^{(i)}}{\alpha_k^{(0)}}\le0
 \end{align*}
 where $C=\max_{1\le i\le N}\sup_{k} \|T_i(x_N^*)-x_N^*\|\|x^{k+1}-x_N^*\|$.
The desired claim can now be deduced from Lemma \ref{lemma:L1}.
\end{proof}
We will now show the rate of convergence for the general case. To study this, let us take the sequences $y_k^{(i)}=T_i(x^{k-1})$ and step-sizes \begin{align}\label{stepsizmulti}
    \alpha_k^{(i)}=\min\{\frac{2}{(1-r)k},1\},~~\beta_k^{(i)}=\frac{\alpha_k^{(i)}-\alpha_{k+1}^{(i)}}{2(2-\alpha_k^{(i)})},~~J=\lfloor{\frac{2}{1-r}}\rfloor
\end{align} }
\textcolor{black}{
\begin{thm}
Let $y_k^{(i)}=T_i(x^{k-1})$  be sequences generated by Algorithm \eqref{al.multi}, where $\alpha_k$ is proposed by \eqref{stepsizmulti}. Then
\begin{align}
{\phi _1}(y_k^{(1)}) - {\phi _1}({x_N^*}) \le \frac{(J+2)(C_S+2C_{T_1}+5C_{x_N^*})\sqrt{JC_{x_N^*}(C_{T_1}+C_{x_N^*})}}{s\sqrt{1-r}(1-r)(k+1)\sqrt{k}}~~ \forall (s,k) \in (0,\frac{1}{L_{f_2}}] \times \NN,
 \end{align}
 where $C_S,C_{T1},C_{x_N^*}$ are the same constants as in Lemma \ref{lem:boundedness} from Lemma \ref{bounded.multi} , $J$ is defined in ,
 Furthermore, one has 
 \begin{align}
 &{\phi _i}(y_{k}^{(i)}) - {\phi _i}({x_N^*}) \le \frac{{{C_{{x_N^*}}}J}}{{2t(1 - r)k}},~~~~~~~~~~~~~~\forall (t,k)\in (0,\frac{1}{L_{f_i}} ]\times\NN,~~i\ne1 
 \\&
\omega ({x^k}) - \omega ({x_N^*}) \le \left({L_\omega } + \frac{{L_\omega ^2}}{{2\mu }}r{C_{{x^*}}}\right)r\sqrt{\frac{C_{x_N^*}(C_{T_1}+C_{x_N^*})J}{\gamma k}}. 
\end{align}
\end{thm}
\begin{proof}
    As \eqref{eq:ratex} and Lemma \ref{lem:multi} and utilizing of Lemma \ref{lem:convergesrate} one can conclude that \[\|x^k-x_N^*\|\le\frac{C_{x_N^*}(C_{T_1}+C_{x_N^*})J}{(1-r)k}\], and the rest of the proof is similar to the one for the trilevel Theorem \ref{thm:rateconvergence}.
\end{proof}}




 \section{Conclusion}

We have shown how to approach a broad class of hierarchical convex optimization problems 
wherein the inner problems optimize the so-called composite functions, i.e., sums of a convex smooth function and a convex non-smooth one, and all but the inner-most problem consider a constraint set composed of  minimizers of another problem.
We have used proximal gradient operators in an 
iterative proximal-gradient algorithm related to ``SAM'' of \cite{Sabach2017}. 
For the first time, we consider diminishing sequences $\alpha_k$ and $\beta_k$ such that the large limit of $\frac{\beta_k}{\alpha_k}$ need not exist. 
\textcolor{black}{The convergence is studied in a number of cases, depending on the relative speed of convergence of $\alpha_k$ and $\beta_k$ and in some cases regularity properties of the problem layers. We showed standard
$\mathcal{O}(\sqrt{\frac{1}{k}})$, $\mathcal{O}(\frac{1}{k})$, $\mathcal{O}(\frac{1}{(k+1)\sqrt{k}})$ rates of convergence for appropriate corresponding quantities. 
Future work can include introducing stochasticity to the problems.} 

\backmatter

\bmhead{Acknowledgments}
Shortly after we have posted our first draft on-line in arxiv, a team 
from the The University of Tokyo, RIKEN, and The Institute of Statistical Mathematics
have submitted their draft \cite{Sato2021}, which considered rather a different method for a closely related problem, albeit without bounding rates of convergence.

\section*{Declarations}

\paragraph*{Funding}
The research leading to these results received funding from OP RDE under Grant Agreement No CZ.02.1.01/0.0/0.0/16\_019/0000765.
This work has received funding from the European Union’s Horizon Europe research and innovation programme under grant agreement No. 101070568. 
This work was supported by Czech Science Foundation (Grant number 22-15524S).

\paragraph*{Conflicts of interest/Competing interests}
The authors have no conflicts of interest to declare that are relevant to the content of this article.


\begin{thebibliography}{}

\bibitem{al1992global} Al-Khayyal, F., Horst, R., Pardalos, P.M.: Global optimization of concave functions subject to quadratic constraints: an application in nonlinear bilevel programming. Ann. Oper. Res. 34, 125-147 (1992)

\bibitem{beck2009fast} Beck, A., Teboulle, M.:  A fast iterative shrinkage-thresholding algorithm for linear inverse problems. SIAM J. Imaging Sci. 2, 183--202  (2009)

\bibitem{beck2014first} Beck, A., Sabach, S.: A first order method for finding minimal norm-like solutions of convex optimization problems.  Math. Program. 147, 25--46 (2014)

\bibitem{bracken1973mathematical} Bialas, W.F., Karwan, M.H.: Mathematical programs with optimization problems in the constraints. Oper. Res. 21, 37-44 (1973)

\bibitem{bauschke2011convex} Bauschke, H.,  Combettes, H., Patrick, L.: Convex analysis and monotone operator theory in Hilbert spaces. Springer, New York  (2011)

\bibitem{bental2009robust} Ben-Tal, Aharon, Laurent El Ghaoui, and Arkadi Nemirovski.: Robust optimization. Vol. 28. Princeton university press (2009)

\bibitem{borwein2017convergence} Borwein, J. M., Guoyin, L., Matthew, T.: Convergence rate analysis for averaged fixed point iterations in common fixed point problems. SIAM J. Optim. 27, 1--33 (2017)

\bibitem{blair1992computational} Blair, C.: The computational complexity of multi-level linear programs. Annals of Operations Research, (34) (1992)


\bibitem{borwein2010convex} Borwein, J. M., Vanderwerff. J.: Convex functions: constructions, characterizations, and counterexamples. Cambridge University Press Cambridge, (172) (2010)


\bibitem{bolte2007lojasiewicz} Bolte, J., Daniilidis, A., Lewis, A.:  The {\L}ojasiewicz inequality for nonsmooth subanalytic functions with applications to subgradient dynamical systems. SIAM J. Optim. , 17(4), 1205-1223 (2007)

\bibitem{bauschke2015linear} Bauschke, H., Noll, D., Phan, HM.:  Linear and strong convergence of algorithms involving averaged nonexpansive operators.  J. Math. Anal. Appl. , 421(1), 1-20(2015)

\bibitem{dempe2007new} Dempe, S., Dutta, J., Mordukhovich, B.S.: New necessary optimality conditions in optimistic bilevel programming. Optimization. 56, 577--604 (2007)

\bibitem{dempe2014necessary} Dempe, S., Dutta, J., Mordukhovich, B.S.: Necessary optimality conditions in pessimistic bilevel programming. Optimization. 56, 505--533 (2014)

\bibitem{alber1994iterative} Dempe, S., Dutta, J., Mordukhovich, B.S.: An iterative method for solving a class of nonlinear operator equations in Banach space.  J. Panamerican. Math. 4, 39--54 (1994)

\bibitem{fisac2015pursuit} Fisac, Jaime F and Sastry, S Shankar.: 
The pursuit-evasion-defense differential game in dynamic constrained environments. 2015 54th IEEE Conference on Decision and Control (CDC). 4549--4556 (2015)



\bibitem{iiduka2011iterative} Iiduka, H.: Iterative algorithm for solving triple-hierarchical constrained optimization problem. J. Optim. Theory. Appl. 148, 580--592  (2011)

\bibitem{yamada2001hybrid} Isao, Y.: The Hybrid Steepest Descent Method for Variational Inequality
Problems oûer the Intersection of the Fixed-Point Sets of Nonexpansiûe Mappings,
Inherently Parallel Algorithms in Feasibility and Optimization and Their Applications, Edited by D. Butnariu, Y. Censor, and S. Reich, North-Holland,
Amsterdam, Holland pp. 473--504, 2001.


\bibitem{lampariello2020explicit} Lampariello, L., Neumann, C., Ricci, J., Sagratella, S., Stein, O.: An explicit Tikhonov algorithm for nested variational inequalities. J. Comput. Appl. 77, 335--350 (2020)

\bibitem{lu2009hybrid} Lu, X.W., Xu, H.K., Yin, X.M.: Hybrid methods for a class of monotone variational inequalities. Nonlinear Anal. 71, 1032--1041 (2009)



\bibitem{moudafi2007krasnoselski} Moudafi, A.: Krasnoselski–Mann iteration for hierarchical fixed-point problems. Inverse. Probl. 23, 1635--1640 (2007)


\bibitem{mainge2007strong} Maing{\'e}, P.E., Abdellatif, M.: Strong convergence of an iterative method for hierarchical fixed-point problems. Pac. J. Optim. 3, 529--538  (2007)

\bibitem{nesterov2003introductory} Nesterov, Y.: Introductory lectures on convex optimization. Springer Science \& Business Media 87  (2003)

\bibitem{solodov2007explicit} Solodov, M.: An explicit descent method for bilevel convex optimization. J. Convex. Anal. 14, 277 (2007)

\bibitem{Sabach2017} Sabach, S., Shtern, S.: A First Order Method for Solving Convex Bilevel Optimization Problems. SIAM J. Optim. 27, 640--660 (2017)


\bibitem{Sato2021} Sato, R., Mirai T., Takeda, A.: A Gradient Method for Multilevel Optimization. Advances in Neural Information Processing Systems 34, 7522--7533 (2021)


\bibitem{senter1974approximating} Senter., HF and Dotson.: Approximating fixed points of nonexpansive mappings. Proceed. Amer. Math. Society, (44) 375--380 (1974)


\bibitem{candler1977multi} Wilfred, W.: Multi-level programming. World Bank (1977)

\bibitem{xu2010viscosity} Xu, H.K.: Viscosity method for hierarchical fixed point approach to variational inequalities.  Taiwan. J. Math. 14, 463--478 (2010)


\bibitem{xu2004viscosity} Xu, H.K.: Viscosity approximation methods for nonexpansive mappings. J. Math. Anal. Appl, (298) 279--291 (2004)

\bibitem{xu2002iterative} Xu, H.K.: Iterative algorithms for nonlinear operators. J. Lond. Math. Soc. 66, 240--256 (2002)


\bibitem{zhang1994problems} Zhang, R.: Problems of hierarchical optimization in finite dimensions. SIAM J. Optim. 4, 521--536 (1994)





\end{thebibliography}
\end{document}